\documentclass[12pt,letterpaper]{article}
	
	\usepackage{graphics,epsfig,graphicx,color}
	\graphicspath{{/EPSF/}{../Figures/}{Figures/}}
	
	\usepackage{amssymb,latexsym}
	
	\usepackage{setspace}
	\usepackage{amsmath}
	
	\usepackage{mathrsfs,amsfonts}
	
	\usepackage{amsthm,amsxtra}
	
	\usepackage[font={small,it}]{caption}
	
	\newcommand\eq[1] {(\ref{#1})}

	\newcommand\fig[1] {\ref{fig:#1}}

	\newcommand\labfig[1] {\label{fig:#1}}
	
	\newtheorem{lemma}{Lemma}[section]

	\newcommand{\bfm}[1]{\mbox{\boldmath ${#1}$}}
	\newcommand{\nonum}{\nonumber \\}
	\newcommand{\beqa}{\begin{eqnarray}}
	\newcommand{\eeqa}[1]{\label{#1}\end{eqnarray}}
	\newcommand{\beq}{\begin{equation}}
	\newcommand{\eeq}[1]{\label{#1}\end{equation}}
	\newcommand{\bpm}{\begin{pmatrix}}
	\newcommand{\epm}{\end{pmatrix}}
	\newcommand{\Grad}{\nabla}
	\newcommand{\Div}{\nabla \cdot}
	\newcommand{\Curl}{\nabla \times}

	\newcommand{\Imag}{\mathop{\rm Im}\nolimits}
	\newcommand{\Tr}{\mathop{\rm Tr}\nolimits}

	\newcommand{\lang}{\langle}
	\newcommand{\rang}{\rangle}
	\newcommand{\Md}{\partial}

	\newcommand{\Ga}{\alpha}
	\newcommand{\Gb}{\beta}
	\newcommand{\Gd}{\delta}
	\newcommand{\Ge}{\epsilon}

	\newcommand{\Gg}{\gamma}

	\newcommand{\Gk}{\kappa}
	
	\newcommand{\Gl}{\lambda}

	\newcommand{\Gt}{\theta}
	
	\newcommand{\Gr}{\rho}
	
	\newcommand{\Gs}{\sigma}

	\newcommand{\Go}{\omega}

	\newcommand{\GO}{\Omega}
	
	\newcommand{\GY}{\Psi}
	
	\newcommand{\BGa}{\bfm\alpha}
	\newcommand{\BGb}{\bfm\beta}
	
	\newcommand{\BGe}{\bfm\epsilon}
	\newcommand{\BGve}{\bfm\varepsilon}
	\newcommand{\BGf}{\bfm\phi}

	\newcommand{\BGm}{\bfm\mu}

	\newcommand{\BGr}{\bfm\rho}
	
	\newcommand{\BGs}{\bfm\sigma}
	
	\newcommand{\BGj}{\bfm\tau}

	\newcommand{\BGG}{\bfm\Gamma}
	\newcommand{\BGL}{\bfm\Lambda}

	\newcommand{\BGY}{\bfm\Psi}

	\newcommand{\CA}{{\cal A}}
	\newcommand{\CB}{{\cal B}}
	\newcommand{\CC}{{\cal C}}
	\newcommand{\CD}{{\cal D}}
	\newcommand{\CE}{{\cal E}}
	
	\newcommand{\CG}{{\cal G}}
	\newcommand{\CH}{{\cal H}}
	
	\newcommand{\CJ}{{\cal J}}
	\newcommand{\CK}{{\cal K}}
	
	\newcommand{\CM}{{\cal M}}

	\newcommand{\CQ}{{\cal Q}}
	\newcommand{\CR}{{\cal R}}
	\newcommand{\CS}{{\cal S}}
	\newcommand{\CT}{{\cal T}}
	\newcommand{\CU}{{\cal U}}
	
	\newcommand{\CW}{{\cal W}}

	\newcommand{\BCC}{{\bfm{\cal C}}}
	\newcommand{\BCD}{{\bfm{\cal D}}}

	\newcommand{\BCS}{{\bfm{\cal S}}}

	\def\Ba{{\bf a}}
	\def\Bb{{\bf b}}
	
	\def\Bd{{\bf d}}
	\def\Be{{\bf e}}
	\def\Bf{{\bf f}}
	\def\Bg{{\bf g}}
	\def\Bh{{\bf h}}
	
	\def\Bj{{\bf j}}
	\def\Bk{{\bf k}}
	
	\def\Bm{{\bf m}}
	\def\Bn{{\bf n}}

	\def\Bq{{\bf q}}
	
	\def\Bs{{\bf s}}
	
	\def\Bu{{\bf u}}
	\def\Bv{{\bf v}}
	\def\Bw{{\bf w}}
	\def\Bx{{\bf x}}
	\def\By{{\bf y}}
	
	\def\BA{{\bf A}}
	\def\BB{{\bf B}}
	\def\BC{{\bf C}}
	\def\BD{{\bf D}}
	\def\BE{{\bf E}}
	\def\BF{{\bf F}}
	\def\BG{{\bf G}}
	
	\def\BI{{\bf I}}
	\def\BJ{{\bf J}}
	\def\BK{{\bf K}}
	\def\BL{{\bf L}}
	\def\BM{{\bf M}}
	\def\BN{{\bf N}}
	
	\def\BP{{\bf P}}
	\def\BQ{{\bf Q}}
	\def\BR{{\bf R}}
	\def\BS{{\bf S}}
	\def\BT{{\bf T}}
	\def\BU{{\bf U}}
	
	\def\BW{{\bf W}}

	\def\BZ{{\bf Z}}
	\def\B0{{\bf 0}}

	\def \RR {{\mathbb R}}
	
	\def \ba {\begin{array}}
	\def \ea {\end{array}}
	\newtheorem {Thm} {Theorem} [section]

	\newtheorem {Adef} [Thm] {Definition}
	
	\newtheorem {Arem} [Thm] {Remark}
	
	\newtheorem {Aexa} [Thm] {Example}
	
	\newtheorem {Anot} [Thm] {Notation}

	\def \refe #1.{(\ref{#1})}
	\def \reff #1.{figure~\ref{#1}}
	\def \refs #1.{section~\ref{#1}}
	\def \refss #1.{subsection~\ref{#1}}
	\def \refD #1.{Definition~\ref{#1}}
	\def \refT #1.{Theorem~\ref{#1}}
	\def \refL #1.{Lemma~\ref{#1}}
	\def \refC #1.{Corollary~\ref{#1}}
	\def \refP #1.{Proposition~\ref{#1}}
	\def \refR #1.{Remark~\ref{#1}}
	\def \refE #1.{Example~\ref{#1}}
	\def \refN #1.{Notation~\ref{#1}}
	\newif\ifPDF
	\ifx\pdfoutput\undefined
	    \PDFfalse
	\else
	    \ifnum\pdfoutput > 0
	        \PDFtrue
	    \else
	        \PDFfalse
	    \fi
	\fi
	
	\ifPDF
	    \usepackage{pdftricks}
	    \begin{psinputs}
	        \usepackage{pstricks}
	        \usepackage{pstcol}
	        \usepackage{pst-plot}
	        \usepackage{pst-tree}
	        \usepackage{pst-eps}
	        \usepackage{multido}
	        \usepackage{pst-node}
	        \usepackage{pst-eps}
	    \end{psinputs}
	\else
	    \usepackage{pstricks}
	\fi
	
	\ifPDF
	    \usepackage[debug,pdftex,colorlinks=true,
	    linkcolor=blue, bookmarksopen=false,
	    plainpages=false,pdfpagelabels]{hyperref}
	\else
	    \usepackage[dvips]{hyperref}
	\fi
	
	\usepackage{fancyhdr}
		\usepackage{color}
\newcommand{\aj}[1]{\textcolor{black}{#1}}
\newcommand{\ajj}[1]{\textcolor{black}{#1}}
\newcommand{\ay}[1]{\textcolor{black}{#1}}
\newcommand{\ayy}[1]{\textcolor{black}{#1}}
	\newenvironment{keywords}
	{\noindent{\bf Key words.}\small}{\par\vspace{1ex}}
	\newenvironment{AMS}
	{\noindent{\bf AMS subject classifications 2000.}\small}{\par}

	\title{Exact relations for Green's functions in linear PDE and boundary field equalities: a generalization of conservation laws}
	
	\author{
	    Graeme W. Milton\thanks{Department of Mathematics, University of Utah, e-mail: milton@math.utah.edu}\;\;, \;Daniel Onofrei\thanks{
	    	        Department of Mathematics,
	    	        University of Houston, TX, e-mail: onofrei@math.uh.edu}}

\begin{document}
	
	\maketitle
	
	\tableofcontents
	
	
	\begin{abstract}
Many physical equations have the form $\BJ(\Bx)=\BL(\Bx)\BE(\Bx)-\Bh(\Bx)$ with source $\Bh(\Bx)$ and fields $\BE$ and $\BJ$ satisfying differential constraints, 
symbolized by $\BE\in\cal E$, $\BJ\in\cal J$ where $\cal E$, $\cal J$ are orthogonal spaces. We show that if $\BL(\Bx)$ takes values in certain nonlinear 
manifolds $\cal M$, and coercivity and boundedness conditions hold, then the infinite body Green's function (fundamental solution) satisfies exact identities. 
The theory also links Green's functions of different problems. The analysis is based on the theory of exact relations for composites, but without assumptions 
about the length scales of variations in $\BL(\Bx)$, and more general equations, such as for waves in lossy media, are allowed. 
For bodies $\Omega$, inside which $\BL(\Bx)\in{\cal M}$, the "Dirichlet-to-Neumann map" (DtN map) giving the response also satisfies exact relations. These boundary 
field equalities generalize the notion of conservation laws: the field inside $\Omega$ satisfies certain constraints, that leave a wide choice in these fields, 
but which give identities satisfied by the boundary fields, and moreover provide constraints on the fields inside the body. A consequence is the following: 
if a matrix valued field $\BQ(\Bx)$ with divergence-free columns takes values within $\Omega$ in a set $\cal B$ (independent of $\Bx$) that lies on a nonlinear manifold, 
we find conditions on the manifold, and on $\cal B$, that with appropriate conditions on the boundary fluxes $\Bq(\Bx)=\Bn(\Bx)\cdot \BQ(\Bx)$ 
(where $\Bn(\Bx)$ is the outwards normal to $\partial\Omega$) force $\BQ(\Bx)$ within $\Omega$ to take values in a subspace $\cal D$. This forces $\Bq(\Bx)$ 
to take values in $\Bn(\Bx)\cdot\cal D$. We find there are additional divergence free fields inside $\Omega$ that in turn generate additional boundary field equalities. 
Consequently, there exist partial Null-Lagrangians, functionals $F(\Bw,\nabla \Bw)$ of a vector potential $\Bw$ and its gradient, that act as null-Lagrangians 
when $\nabla \Bw$ is constrained for $\Bx\in\Omega$ to take values in certain sets $\cal A$, of appropriate non-linear manifolds, and when $\Bw$ satisfies 
appropriate boundary conditions. The extension to certain non-linear minimization problems is also sketched.
	   
	\end{abstract}
	
	
	
	

	\section{Introduction}
\setcounter{equation}{0}
Many important linear equations of physics in an inhomogeneous medium of infinite extent in $\RR^d$
can be written as a system of second-order linear partial differential equations:
\beq \sum_{i=1}^d\frac{\Md}{\Md x_i}\left(\sum_{j=1}^d\sum_{\Gb=1}^mL_{i\Ga j\Gb}(\Bx)\frac{\Md u_\Gb(\Bx)}{\Md x_j}\right)=f_\Ga(\Bx), \quad\Ga=1,2,\ldots,m,\eeq{0.000}
for the $m$-component potential $\Bu(\Bx)$ given the $m$-component source term $\Bf(\Bx)$. If the integral of $\Bf(\Bx)$ over $\RR^d$ is zero, these can be reexpressed as 
\beq J_{i\Ga}(\Bx)=\sum_{j=1}^d\sum_{\Gb=1}^mL_{i\Ga j\Gb}(\Bx)E_{j\Gb}(\Bx)-h_{i\Ga}(\Bx),\quad E_{j\Gb}(\Bx)=\frac{\Md u_\Gb(\Bx)}{\Md x_j},\quad \sum_{i=1}^d\frac{\Md J_{i\Ga}(\Bx)}{\Md x_i}=0,
\eeq{b.0}
where, counter to the usual convention, we find it convenient to let the divergence act on the first index of $\BJ$, and to let the gradient in $\BE=\Grad\Bu$ be associated
with the first index of $\BE$, and $\Bh(\Bx)$ is chosen so
\beq \sum_{i=1}^d\frac{\Md h_{i\Ga}(\Bx)}{\Md x_i}=f_\Ga(\Bx). \eeq{b.00}
Assuming we are looking for solutions where $\BE(\Bx)$ and $\BJ(\Bx)$ are square-integrable \aj{in $\RR^d$}, integration by parts shows that
\beq \int_{\RR^d} \sum_{i=1}^d\sum_{\Ga=1}^m J_{i\Ga}(\Bx)E_{i\Ga}(\Bx)\,d\Bx=0. \eeq{b.000}
Thus $\BE(\Bx)$ and $\BJ(\Bx)$ belong to orthogonal spaces: $\CE$ the set of square-integrable fields $\BE(\Bx)$ such that $\BE=\Grad\Bu$ 
for some $m$ component potential $\Bu$, and $\CJ$ the set of square-integrable fields $\BJ(\Bx)$ such that $\Div\BJ=0$. With these definitions, the equations \eq{b.0}
take the equivalent, more abstract, form 
\beq \BJ(\Bx)=\BL(\Bx)\BE(\Bx)-\Bh(\Bx),\quad \BJ\in\CJ,\quad \BE\in\CE,\quad \Bh\in\CH.
\eeq{I.1}
where $\CH=\CE\oplus\CJ$ consists of square integrable $d\times m$ matrix-valued fields.
When $m=1$ we can interpret these equations as conductivity equations, with $\BJ$ as the current field, $\Bf=\Div\Bh$ as a source of current, $\BE=-\Grad V$
as the electric field, and $V=-u$ as the electrical potential. Then the $\Gs_{ij}(\Bx)=L_{i1j1}(\Bx)$ are the elements of the conductivity tensor field $\BGs(\Bx)$.

\aj{As shown in the appendix (see also \cite{Milton:2002:TOC} Chap.2, \cite{Milton:2016:ETC} Chap.1, \cite{Milton:2013:SIG,Milton:2015:ATS} and the appendix of \cite{Milton:2018:NRF})}
this structure \eq{I.1} is suitable for a multitude of 
additional linear physical equations too,
including wave equations. They can be formulated in a Hilbert space $\CH$ of
square integrable fields \aj{in $\RR^d$} taking values in some tensor space $\CT$, where $\CH$ can be split into two orthogonal subpaces
$\CE$ and $\CJ$, i.e., $\CH=\CE\oplus\CJ$. This splitting is typically such that the operator $\BGG_1$ that projects onto $\CE$
is local in Fourier space, i.e., if $\BE=\BGG_1\BA$ then the Fourier components $\widehat{\BE}(\Bk)$ and $\widehat{\BA}(\Bk)$
of $\BE$ and $\BA$ are related via $\widehat{\BE}(\Bk)=\BGG_1(\Bk)\widehat{\BA}(\Bk)$ for some operator $\BGG_1(\Bk)$ that projects onto a subspace $\CE_{\Bk}\subset\CT$.
\aj{The appendix gives examples of fields} $\BE\in\CE$ having some components, that beside those involving $\Grad\Bu$, just involve $\Bu$ alone: examples are the acoustic,
electrodynamic, or elastodynamic wave equations in possibly lossy (energy absorbing) media \cite{Milton:2009:MVP,Milton:2010:MVP,Milton:2016:ETC} at constant (possibly complex)
frequency). The lossy nature of the moduli at constant frequency ensures the coercivity we need for well posedness.
 Also the fields $\BE\in\CE$ could have higher order gradients:
a classic example is the Kirchoff plate equation  (see, e.g.,\cite{Milton:2002:TOC} Sect. 2.3, and references therein)
where one takes $\BE(\Bx)$ to be the \ayy{linearized} plate curvature $\BE=\Grad\Grad u$, in which $u(\Bx)$ is the (infinitesimal)
vertical deflection of the plate. Furthermore, some components of the fields $\BE\in\CE$ are not necessarily derivatives of potentials but could also involve, say,
divergence-free vector fields (and the corresponding components of the fields $\BJ\in\CJ$ would then be gradients of potentials). Such mixed formulations
are useful in quasistatic and wave equations in lossy media when one wants to reformulate the problem so that $\BL(\Bx)$  is real and positive definite
\cite{Cherkaev:1994:VPC,Milton:2009:MVP,Milton:2010:MVP,Milton:2016:ETC}. \aj{Again, examples are given in the appendix.}
 
To begin we are considering an inhomogeneous medium of infinite extent in $\RR^d$,
where the material moduli are contained in a tensor $\BL(\Bx)$. 
Some type of boundedness and coercivity constraints on $\BL(\Bx)$ are usually needed to ensure that the equations \eq{I.1} always have a  unique {\color{black}{weak}}
solution for $\BE(\Bx)$ (and hence $\BJ(\Bx)$) for any given source field $\Bh(\Bx)$ with say compact support. {\color{black} {The concept of weak solution implies that, further throughout the paper, unless more regularity is specified, the tangential component of the $\BE(\Bx)$ along the boundary can be viewed in the $H^{\frac{1}{2}}$ sense and the normal component of $\BJ(\Bx)$ over the boundary in the $H^{-\frac{1}{2}}$ sense. }}

Then 
in the governing equations \eqref{I.1} the uniquely determined field $\BE$ depends linearly on the source term $\Bh$. \aj{So assuming
$\BE$ depends continuously on $\Bh$, as \ajj{it} should in any physical problem of interest,
the Schwartz kernel theorem implies that we can informally write}
     \begin{equation}
    \label{12.1}
     \BE(\Bx)=\int_{\mathbb{R}^d}\BG(\Bx,\Bx')\Bh(\Bx')~d\Bx',
    \end{equation}
where the integral kernel $\BG(\Bx,\Bx')$ (possibly a generalized function) is the Green's function for the problem, that depends on both $\Bx$ and $\Bx'$, 
and not just on $\Bx-\Bx'$, because the medium is inhomogeneous. 
\ayy{
\begin{Arem}
	\label{smooth-assumption}
	In general the existence of a continuous Green Function for problem \eqref{I.1} may be difficult to prove. Thus, for the sake of clarity of exposition, further in the paper $\delta$ will denote a smooth approximation of the Dirac delta distribution and $\BG$ will denote an approximate Green function (still called the Green's function), i.e. the solution of problem \eqref{I.1} with source $\delta$ (i.e., a smooth approximation of the Dirac delta distribution).
\end{Arem}
}
The Green's function can also be considered as a linear map $\BG:\CH \rightarrow \CH$, or more specifically, $\BG:\CH  \rightarrow \CE$.
The main objective in this paper is to show that when $\BL(\Bx)$ is constrained to take values in certain non-linear manifolds $\CM$
then the Green's function kernel $\BG(\Bx,\Bx')$ satisfies some exact identities for every $\Bx\ne\Bx'$. The manifold $\CM$, with dimension $k_0$, need not have
codimension one. 

When we say $\CT$ is a tensor space we mean that there is a natural inner product $\lang\BA,\BB\rang_\CT$ between any $\BA, \BB\in\CT$ and
for every $d$-dimensional rotation $\BR$ there exists an associated linear operator $\BQ(\BR)$ acting on $\CT$ such that $\lang\BQ(\BR)\BA,\BQ(\BR)\BB\rang_\CT=\lang\BA,\BB\rang_\CT$
for all $\BA, \BB\in\CT$. Thus, for example, $\CT$ could consist of vectors that have a combination of scalars, vectors, second-order tensors, or higher order tensors
(or even tensors of ``half integer'' order, like spins in quantum mechanics) as elements. However, the tensorial nature of $\CT$ is rather moot in this paper as we are not 
concerned with the action of rotations on elements of $\CT$. Indeed, \eq{b.0} with $m=d$ could be regarded as the linear elasticity equations with $L_{i\Ga j\Gb}$ being 
the 4th-order elasticity tensor $C_{i\Ga j\Gb}$ (that annihilates any antisymmetric component of $\BE$) 
and with $\BJ(\Bx)$ and $\BE(\Bx)$ being the second order tensor stress and displacement gradient fields. But mathematically
it is the same problem when the components of $\Bu$ represent different physical scalar potentials, e.g., such as temperature, electrical potential, 
and pressure, that are invariant under rotations. Then for any $\Bx$, $\BJ(\Bx)$ and $\BE(\Bx)$ are not second order tensors, but rather triplets of vector fields. 

Our paper also presents a broad theory of boundary field equalities that generalize the notion of a conservation law. \aj{These boundary field equalities imply, for example,
that the "Dirichlet-to-Neumann map" (DtN map) governing the response of inhomogeneous bodies satisfies certain exact identities when the tensor field $\BL(\Bx)$ inside $\GO$ takes values in 
a certain nonlinear manifold $\CM$. These identities generalize the  exact identities satisfied by the effective tensor $\BL_*$ in the theory of exact relations for composites
when $\BL(\Bx)$ inside the period cell takes values in $\CM$.}

The classic conservation law says that if a vector field $\Bu(\Bx)$ with $d$ components satisfies
$\Div\Bu=0$ inside a body $\GO$, then the integral of $\Bu\cdot\Bn$ over the 
surface $\partial\GO$ of $\GO$ is zero: here $\Bn$ denotes the outward normal of $\GO$. This naturally leads
to the question: can one make other assumptions about the fields inside a body (still leaving many degrees of freedom in the choice of these fields) that imply exact ``boundary field equalities'' among the fields at the boundary for suitable boundary conditions? \aj{Of course, these boundary conditions
should not be such that they trivially imply the boundary field equalities, independent of any assumption about the fields inside the body. \it{We emphasize that, in general, our boundary field equalities do not result from
integration by parts, but rather arise through algebraic properties of the underlying operators. Thus it is an entirely new idea to obtaining identities satisfied by the boundary fields.}}

A divergence-free field satisfies a differential constraint, but additional
algebraic constraints are also possible. An example of the latter type of boundary field equality, discussed in \cite{Milton:2016:ETC} Sect. 1.5,
and implicit in the work of Milgrom \cite{Milgrom:1990:LRG} (see also \cite{Milton:2002:TOC} Chap. 6, and references therein) is the following one. Consider in a body $\GO$
the \aj{primary} equations \eq{b.0} with no source term, i.e., $h_{ij}(\Bx)=0$ for all $i,j$, and with $L_{i\Ga j\Gb}(\Bx)=\Gd_{ij}A_{\Ga\Gb}(\Bx)$.
Assume that $\GO$ contains just two phases in any configuration where the $m\times m$ matrix valued field $\BA(\Bx)$ takes
 the value $\BA^{(1)}$ in phase 1 and $\BA^{(2)}$ in phase 2, in which $\BA^{(1)}$ and $\BA^{(2)}$ are real, symmetric, positive definite matrices.
Associated with this problem are boundary fields: the vector potential $\Bu(\Bx)$ on $\Md\GO$ (that may be obtained by integrating over the surface
the tangential values of $\BE(\Bx)$) and the vector-valued flux $\Bn\cdot\BJ$. The key observation is that there exists a congruence transformation that
simultaneously diagonalizes $\BA^{(1)}$ and $\BA^{(2)}$, i.e. a matrix $\BW$ such that $\BW\BA^{(1)}\BW^{T}$ and $\BW\BA^{(2)}\BW^{T}$ are simultaneously
diagonal. To do this we choose $\BW=\BQ(\BA^{(2)})^{-1/2}$ where $\BQ$ satisfying $\BQ\BQ^T=\BI$ is taken to 
diagonalize $(\BA^{(2)})^{-1/2}\BA^{(1)}(\BA^{(2)})^{-1/2}$. Then one obtains an equivalent set of decoupled conductivity equations:
\beq \widetilde{J}_{i\Ga}(\Bx)=\Gs_{\Ga}(\Bx)\widetilde{E}_{i\Ga}(\Bx),\quad \widetilde{E}_{i\Ga}(\Bx)=\frac{\Md \widetilde{u}_\Ga(\Bx)}{\Md x_i},\quad \sum_{i=1}^d\frac{\Md \widetilde{J}_{i\Ga}(\Bx)}{\Md x_i}=0,
\eeq{b.0a}
indexed by $\Ga=1,2,\ldots,m$, with 
\beqa \widetilde{\BJ}(\Bx)& = & \BJ(\Bx)\BW^T,\quad \widetilde{\BE}(\Bx)=\BE(\Bx)\BW^{-1}, \quad \widetilde{\Bu}(\Bx)=(\BW^T)^{-1}\Bu(\Bx),\nonum
\Gs_\Ga(\Bx) &= & \{\BQ(\BA^{(2)})^{-1/2}\BA^{(1)}(\BA^{(2)})^{-1/2}\BQ^T\}_{\Ga\Ga} \mbox{ for $\Bx$ in phase 1 }, \nonum
             &= & 1 \mbox{ for $\Bx$ in phase 2}.
\eeqa{b.0b}
It is then clear that if the boundary values of $\Bu(\Bx)$ are prescribed so that $\widetilde{\Bu}(\Bx)=(\BW^T)^{-1}\Bu(\Bx)$ only has one non-zero component, then certainly
the flux $\Bn\cdot\widetilde{\BJ}(\Bx)=\Bn\cdot\BJ(\Bx)\BW^T$ will only have one matching non-zero component. In other words 
the flux $\Bn\cdot\BJ(\Bx)$ is of the form $\Ga(\Bx)\Bv$ where only the scalar $\Ga(\Bx)$ varies on the surface of $\GO$. These constraints on the flux $\Bn\cdot\BJ(\Bx)$
for the prescribed $\Bu(\Bx)$ are an example of a boundary field equality. Another example of a boundary field equality is that given in \cite{Thaler:2014:EDV}
for an elastic body containing two isotropic elastic phases having the same shear modulus, where the boundary field equality involves the volume fraction
(and thus may be used in an inverse way to determine this volume fraction). 

These simple examples serve to give an idea of boundary field equalities, but the general theory developed here goes far beyond them.
Rather than say considering inside $\GO$ the constitutive law $\BJ(\Bx)=\BL(\Bx)\BE(\Bx)$ with the constraint that $\BL(\Bx)\in\CM$, 
we may eliminate $\BL(\Bx)$ and just view these relations as a non-linear local constraint on the fields. Then, 
we obtain results like the following. Suppose we are given a $d\times s$ 
matrix valued field $\BQ(\Bx)$ such that $\Div\BQ=0$ (so each column of $\BQ(\Bx)$ represents a single divergence-free vector field). 
Let  $\Bq(\Bx)=\Bn(\Bx)\cdot\BQ(\Bx)$ be the associated $s$-component flux $\Bq(\Bx)=\Bn(\Bx)\cdot\BQ(\Bx)$ at the boundary of $\GO$ (where $\Bn(\Bx)$ is the outwards normal on $\Md\GO$). With $\BQ(\Bx)$ constrained to take values in a
subset $\CB$ of some $r$-dimensional non-linear manifold (where $\CB$ does not depend on $\Bx$), we find conditions on $\CB$ and
non-local linear constraints on the boundary flux $\Bq$ (specified in Section 8) which forces the non-uniquely determined field $\BQ(\Bx)$ inside $\GO$ to lie in a subspace $\CD$,
and hence which forces $\Bq(\Bx)$ to lie in $\Bn(\Bx)\cdot\CD$ (where $\Bn(\Bx)\cdot\CD$ is obtained by applying
$\Bn(\Bx)\cdot$ to each element of $\CD$). \aj{Of course, it should not be the case that $\CB\subset\CD$, since otherwise the result would be trivial.}

Alternatively, we may express $\BQ$ in terms of the elements gradient of $t$-component potential $\Bw$, and we obtain results like the following. 
Suppose $\Grad\Bw(\Bx)\in\CA$ for all $\Bx$, where $\CA$ is a subset of a non-linear manifold (that does not depend on $\Bx$),
then we find conditions on $\CA$ and non-local linear constraints on the surface potential $\Bw(\Bx)$, $\Bx\in\Md\GO$,  (specified in Section 8)
which forces the non-uniquely determined field $\Grad\Bw(\Bx)$ inside $\GO$ to lie in a subspace $\CC$ , and this then places restrictions
on the tangential derivatives of the surface potential.  \aj{Again, it should not be the case that $\CA\subset\CC$, since otherwise the result would be trivial.}
To better understand the significance of the constraint that inside $\GO$, 
$\Grad\Bw(\Bx)\in\CC$, let $\BN$ be any $d\times  t$ matrix normal to the space $\CC$ ,
i.e. such that $\Tr(\BC\BN^T)=0$ for all $\BC\in\CC$ . Then $\Tr[\Grad\Bw(\Bx)\BN^T]=0$ which implies $\Div[\BN\Bw(\Bx)]=0$, i.e.
$\BN\Bw(\Bx)$ is a divergence-free field.  We deduce that
\beq \int_{\Md\GO}\Bn(\Bx)\cdot[\BN\Bw(\Bx)]~dS=0,
\eeq{b.0e}
which implies the additional boundary field equality: 
\beq \int_{\Md\GO}\Bn(\Bx)\otimes\Bw(\Bx)~dS\in \CC.
\eeq{b.0f} 
An explanation of why there exist such subsets $\CB$ and $\CA$ of non-linear manifolds with this property, and a prescription for obtaining them and 
the appropriate boundary conditions, will be given in Section 8.

\ayy{Somewhat related questions have been the focus of attention in the homogenization community. Luc Tartar \cite{Tartar:1979:CCA}
raised (in a more general setting) essentially this fundamental question: if one has a sequence of fields $\Bu_\Ge(\Bx)$ such that $\Grad\Bu_\Ge(\Bx)$ 
takes values in a set $\CA$,
then what is the range of values that the weak limits of $\Grad\Bu_\Ge(\Bx)$ can take? (Alternatively, if one has a sequence of
fluxes $\BQ_\Ge(\Bx)$ such that $\Div\BQ_\Ge(\Bx)=0$ and $\BQ_\Ge(\Bx)$ takes values in a set $\CB$ and converges weakly to $\BQ_0(\Bx)$,
then what is the range of values that $\BQ_0(\Bx)$ can take?).
 For a sample of work addressing this thorny problem, see, for example, the papers \cite{Zhang:1998:SQH,Muller:2003:CIL,Faraco:2008:TCL} 
and references therein. 
Here we are tackling the question of whether, for certain sets $\CA$, one can deduce additional constraints on each 
element $\Grad\Bu_\Ge(\Bx)$ in the entire sequence,
namely that $\Grad\Bu_\Ge(\Bx)\in\CC$, when $\Bu_\Ge(\Bx)$ satisfies appropriate boundary conditions.
This is important as  the homogenization approach is not suitable in applications where there is no separation of length scales. In cases 
where it is applicable we deduce the non-trivial result that the weak limit of $\Grad\Bu_\Ge(\Bx)$ lies in $\CC$ (that is, essentially, a corollary of the theory
of exact relations for composites). We emphasize that our boundary field equalities 
are generally not simply an application of integration by parts, not even at a qualitative level. Rather they follow from algebraic identities.}

We mention too, that beyond boundary field equalities there are boundary field inequalities
and some of these go beyond just using convexity and the divergence theorem \cite{Milton:2013:SIG,Milton:2015:ATS}.

\ayy{Although we do not explicitly address this in the paper, we mention here that} our results on boundary field equalities enable us to introduce and give examples of what we call ``partial null-Lagrangians''. Null Lagrangians include, for example,
 functions $F(\Bx,\Bw,\Grad\Bw)$ for which the corresponding integral
\beq \CW(\Bw)=\int_\GO F(\Bx,\Bw(\Bx),\Grad\Bw(\Bx))~d\Bx,
\eeq{b.0g}
has the property that $\CW(\Bw_0+\BGf)=\CW(\Bw_0)$ for any choice of $\BGf\in C^\infty_0(\GO)$ and for any choice of $\Bw_0\in C^1(\overline{\GO})$. Null Lagrangians 
of this form have been completely characterized by Olver and Sivaloganathan \cite{Olver:1988:SNL}. It is well known (see the references in \cite{Olver:1988:SNL})
that when $F(\Bx,\Bw,\Grad\Bw)=\GY(\Grad\Bw)$, then $\GY(\Grad\Bw)$ is a null-Lagrangian if and only if it is an affine combination of subdeterminants of $\Grad\Bw$
of all orders. Other classes of null-Lagrangian have been characterized by Murat \cite{Murat:1978:CPC,Murat:1981:CPC,Murat:1987:SCC} 
(see also Pedregal \cite{Pedregal:1989:WCW}).
Null-Lagrangians are also instrumental in the construction of polyconvex functions and play a fundamental role in the calculus of variations and in establishing the existence and uniqueness of minimizers to large classes
of "energy functions" (see, for example, \cite{Ball:1977:CCE}, the recent review \cite{Benesova:2017:WLS}, and references therein). Additionally they are
an important tool for establishing
bounds on the effective moduli of composite materials, through the ``translation method'', or equivalently, the method of ``compensated compactness'' 
(as summarized in the books \cite{Cherkaev:2000:VMS,Torquato:2001:RHM,Milton:2002:TOC,Allaire:2002:SOH,Tartar:2009:GTH}). 
In the liquid crystal community
it is well known that if a $t$-component vector field $\Bw(\Bx)$ takes values in $\CA$, where $\CA$ consists of vectors of unit length, so that $|\Bw(\Bx)|=1$,
then $(\Grad\Bw)\cdot\Bw=\Grad(\Bw\cdot\Bw)=0$. For any given $d$-component vector field $\Ba(\Bx)$, the function $F(\Bx,\Bw,\Grad\Bw)=\Ba(\Bx)\cdot(\Grad\Bw)\cdot\Bw$ 
is then an example of what we call a partial null-Lagrangian: its integral can be exactly computed under the constraint
that $\Bw(\Bx)\in\CA$ for all $\Bx\in\GO$ once one knows the boundary fields (and, in this example, the integral is zero and independent of the boundary fields). 

More generally,
we call $F(\Bx,\Bw,\Grad\Bw)$ a partial null-Lagrangian on a subset $\CA$ (independent of $\Bx$)  of a nonlinear manifold
in the space of pairs of $t$-component vectors and $d\times t$ matrices, if for every
$\BGf\in C^\infty_0$ and $\Bw_0\in C^1(\overline{\GO})$ satisfying 
\beq (\Bw_0(\Bx),\Grad\Bw_0(\Bx))\in \CA,\quad (\Bw_0(\Bx)+\BGf(\Bx),\Grad(\Bw_0(\Bx)+\BGf(\Bx)))\in\CA,\quad\forall \Bx\in \GO,
\eeq{b.0h}
and with the surface fields $\Bw_0(\Bx)$, $\Bx\in\Md\GO$, satisfying appropriate non-local boundary conditions, one has
$\CW(\Bw_0+\BGf)=\CW(\Bw_0)$.
With $\CA$, and the boundary conditions on $\Bw_0$ chosen so it forces $\Grad(\Bw_0+\BGf)$ to lie in a subspace $\CC$, we obtain functions 
$F(\Bw,\Grad\Bw)$ that are partial null-Lagrangians, but not null-Lagrangians. Specifically, since
$\BN\Bw(\Bx)$ is a divergence-free field [see the text preceeding \eq{b.0e}], an obvious partial null-Lagrangian is any component of the $t$ component function,
\beq \BF(\Bw,\Grad\Bw)=[\BN\Bw]\cdot\Grad\Bw, \eeq{b.0i}
and we have
\beq  \int_\GO \BF(\Bw(\Bx),\Grad\Bw(\Bx))~d\Bx=\int_{\Md\GO}[\Bn(\Bx)\cdot\BN\Bw(\Bx)]\Bw(\Bx)~dS.
\eeq{b.0j}
In two dimensions, if $\BN_1$ and $\BN_2$ are two $2\times  t$ matrices normal to the space $\CC$, then the fact that $\Div[\BN_1\cdot\Bw(\Bx)]=0$ 
allows us to find a potential $W_1(\Bx)$ such that 
\beq \BR^\perp\BN_1\cdot\Bw(\Bx)=\Grad W_1, \quad \BR^\perp=\bpm 0 & 1 \cr -1 & 0 \epm.
\eeq{b.0k}
Furthermore, $\Bn(\Bx)\cdot[\BN_1\Bw(\Bx)]$ gives us the tangential derivatives of $\Grad W_1(\Bx)$ , which when integrated gives the surface
potential $W_1(\Bx)$, $\Bx\in\Md\GO$.  Then
\beq F(\Bx,\Bw,\Grad\Bw)=[\BN_2\Bw(\Bx)]\cdot\BR^\perp\BN_1\Bw(\Bx) \eeq{b.0l}
is a partial null-Lagrangian and we have
\beq  \int_\GO F(\Bw(\Bx),\Grad\Bw(\Bx))~d\Bx=\int_{\Md\GO}[\Bn(\Bx)\cdot\BN_2\Bw(\Bx)]W_1(\Bx)~dS.
\eeq{b0m}
We remark that, in the context of this paper, the constraint that $\BQ(\Bx)\in\CB$, or equivalently that $\Grad\Bw\in\CA$, is automatically satisfied when there are appropriate materials inside
$\GO$ with a bounded and coercive tensor field $\BL(\Bx)\in\CM$ for all $\Bx\in\GO$. Then, if the appropriate boundary conditions are satisfied, the partial null-Lagrangians place
integral constraints on the fields inside $\GO$.

As our work has as its basis the theory of exact relations for composite materials let us briefly review this.
\section{A brief review of exact relations in composites}
\setcounter{equation}{0}
In this setting, one typically starts with a tensor field $\BL(\By)$ that is periodic in $\By$
and which is a linear map from $\CT$ to $\CT$, where $\CT$ is some  $q$-dimensional \ayy{inner product space}. Here $\BL(\By)$
may represent the conductivity tensor, dielectric tensor, elasticity tensor, or a wealth of other physical tensor fields 
(see, for example, \cite{Milton:2002:TOC} Chap. 12).
In homogenization
theory one often considers a body $\GO$ filled by a material having tensor field $\BL(\Bx/\Ge)$ and in the limit $\Ge\to 0$ the body often responds 
to external fields (that are independent of $\Ge$)
as if it were filled with a homogeneous medium with tensor $\BL_*$ that is known as the effective tensor of the medium. In many problems the
problem of determining $\BL_*$ can be formulated as a problem in the abstract theory of composites. The setting is a Hilbert space $\CH$, say of
periodic fields that are square integrable in the unit cell of periodicity and which take values in $\CT$. 
It has a splitting into three orthogonal spaces $\CH=\CU\oplus\CE\oplus\CJ$. For example, in the conductivity problem $q=d$, 
$\CU$ is the space of $d$-dimensional vector fields that are constant (independent of $\By$, where $\By$ can be thought of as a microscale spatial
coordinate) and 
$\CE$ represents gradients of periodic scalar valued potentials, while $\BJ$ denotes those periodic fields that have zero divergence and zero
average value over the unit cell. To determine the effective tensor $\BL_*$ one prescribes a field $\BE_0\in\CU$ and solves the equations
\beq \BJ_0+\BJ=\BL(\BE_0+\BE), \mbox{ with }\BJ_0\in\CU,\quad\BE\in\CE,\quad\BJ\in\CJ, \eeq{0.0}
where the action of $\BL:\CH \rightarrow \CH$ is defined by $\BB=\BL\BA$ with $\BB(\By)=\BL(\By)\BA(\By)$ (i.e., $\BL$ acts locally in space).
Of course $\BL$ needs to be such that these equation have a unique solution for $\BE$ (hence uniquely giving $\BJ_0$ and $\BJ$) for any $\BE_0\in\CU$.
Clearly $\BJ_0$ depends linearly on $\BE_0$ and it is this linear relation that defines the effective tensor: $\BJ_0=\BL_*\BE_0$. This formulation
which stems from ideas in \cite{Kohler:1982:BEC,DellAntonio:1986:ATO} was crystallized in \cite{Milton:1987:MCEa,Milton:1990:CSP}, see also \cite{Milton:2002:TOC} Chap. 12.

In the field of composites there are a myriad of results on what are known as exact relations:
microstructure independent formulae satisfied by effective moduli. A canonical example is Dykhne's result \cite{Dykhne:1970:CTD} for 2-d 
conductivity that if the determinant of the local (anisotropic) conductivity tensor $\BL(\Bx)=\BGs(\Bx)$ is constant,
then the effective conductivity tensor $\BL_*=\BGs_*$ has the same determinant. 
More generally, as illustrated in Fig.\fig{1}, in the theory of exact relations, one wants to find non-linear manifolds $\CM$ (of dimension less than $q^2$)
in the space $L(\CT)$ of linear maps $\CT\to\CT$
such the effective tensor $\BL_*$ lies in $\CM$ whenever $\BL(\By)$ lies in $\CM$ for all $\By$ (and generally $\BL(\By)$ also satisfies some sort
of boundedness and coercivity properties necessary to ensure that $\BL_*$ exists and is unique). In the two-dimensional conductivity example
$\CM$ consists of $2\times 2$ matrices $\BGs$ such that $\det(\BGs)=c$, where $c$ is a constant parameterizing the manifold.
\begin{figure}[!ht]
\centering
\includegraphics[width=0.8\textwidth]{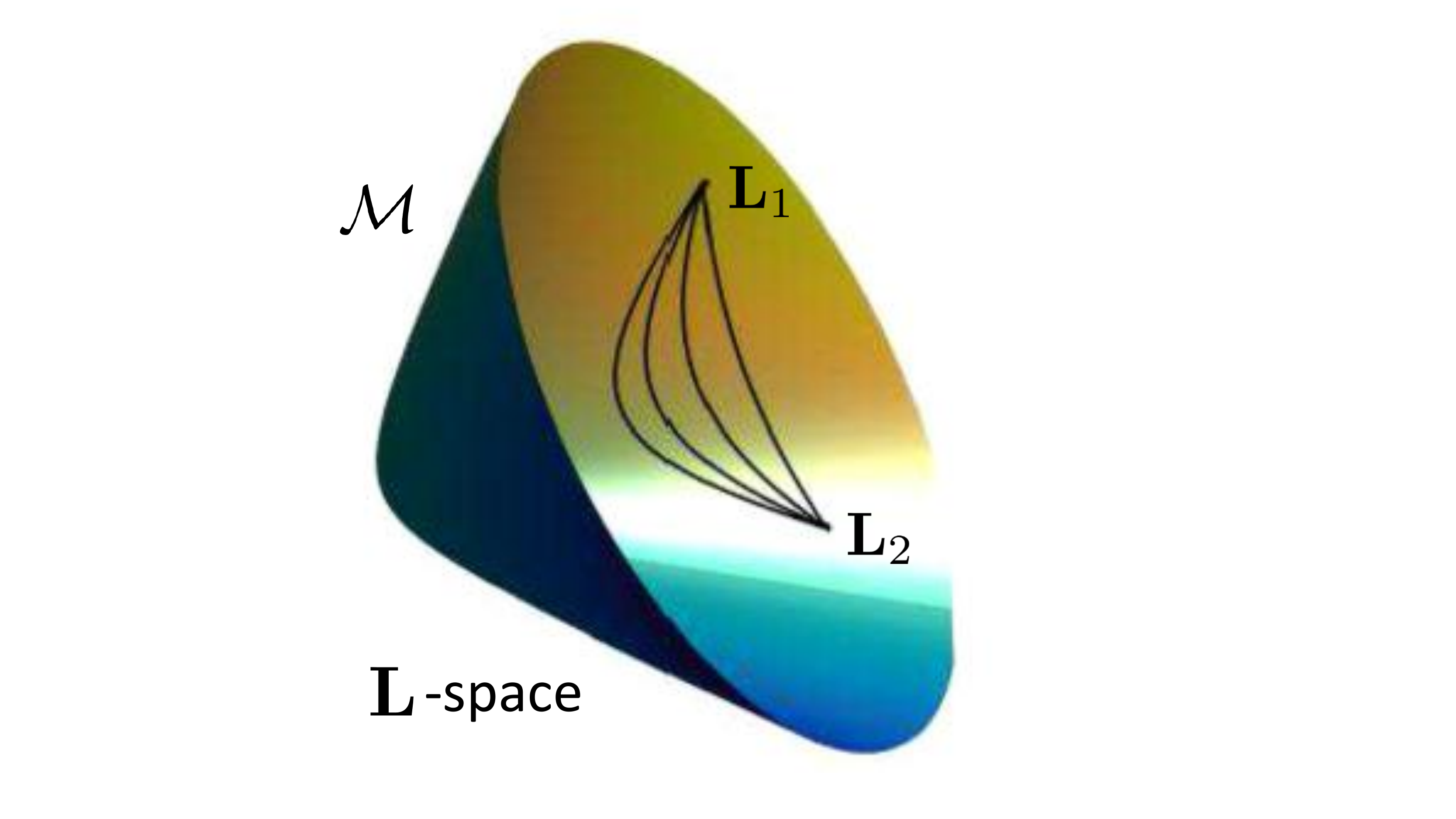}
\caption{\aj{The central goal of the general theory of exact relations for composites is to identify manifolds $\CM$ in tensor space such that if a periodic tensor field $\BL(\Bx)\in\CM$ for all $\Bx$,
(and $\BL(\Bx)$ satisfies boundedness and coercivity conditions that ensure the effective tensor $\BL_*$ exists), then necessarily $\BL_*\in\CM$. Such manifolds are stable under homogenization, 
and hence under lamination. In particular, if one takes two materials with tensors $\BL_1, \BL_2\in\CM$, and layers them together in direction $\Bn$ then the resultant effective tensor $\BL_*$ must also lie in the manifold $\CM$.
Varying the volume fractions occupied by the two materials gives a trajectory that must be confined to $\CM$. The figure shows four trajectories associated  with four different directions of lamination:
$\Bn=\Bn_1,\Bn_2,\Bn_3,$ and $\Bn_4$. \ay{The figure, adapted from figure 4.1 of \protect\cite{Grabovsky:2016:CMM}(``\copyright IOP Publishing. Reproduced with permission. All rights reserved'')}, shows the manifold $\CM$ associated with $2\times 2$ \ayy{symmetric} matrices having constant determinant, corresponding to the Dykhne  \protect\cite{Dykhne:1970:CTD} exact relation for two-dimensional conductivity.}}
\labfig{1}
\end{figure}

The general theory of exact relations was founded by Grabovsky \cite{Grabovsky:1998:EREa}, and his insight is summarized in Fig.\fig{2}).
He realized that if an exact relation held for all composites
then it must certainly hold for layered geometries. For lamination with a vector $\Bn$ perpendicular to the layers, so that $\BL(\By)$ is just a function of 
the single variable $\Bn\cdot\By$, it is convenient to introduce the fractional linear transformation   provided by \cite{Milton:1990:CSP,Zhikov:1991:EHM}
\beq W_{\Bn}(\BL)=[\BI+(\BL-\BL_0)\BGG(\Bn)]^{-1}(\BL-\BL_0),
\eeq{0.00aa}
in which $\BGG(\Bn)$ is a certain tensor dependent on $\BL_0$ and the lamination direction $\Bn$.  
$W_{\Bn}(\BL)$ applied to the local tensor field $\BL(\By)$ and effective tensor $\BL_*$
gives a new tensor field $\BK_\Bn(\By)=W_{\Bn}(\BL(\By))$ and $\BK^*_\Bn=W_{\Bn}(\BL_*)$ that are related simply by a linear average $\BK^*_\Bn=\langle\BK_\Bn\rangle$ (where the angular brackets
denote a volume average of $\BK_\Bn(\By)$ over its unit cell of periodicity). Thus the relation
\beq \BL_*=W_{\Bn}^{-1}(\langle W_{\Bn}(\BL)\rangle), \eeq{0.0a}
determines the effective tensor $\BL_*$. Furthermore one can always choose $W_{\Bn}$ so that $W_{\Bn}(\BL_0)=0$ for some $\BL_0\in\CM$.
(There are other linear lamination formula \cite{Backus:1962:LWE,Tartar:1979:ECH}, but it is unclear if the general theory of exact relations 
can be developed using them, or their generalizations).
Therefore the exact relation in these new
coordinates must be a linear relation: $\BK^*_\Bn\in\CK_\Bn$ when $\BK_\Bn(\By)\in\CK_\Bn$ where the tensor subspace $\CK_\Bn$ defines the exact relation: $\CM=W_{\Bn}^{-1}(\CK_\Bn)$.
Thus $\CK_\Bn$ has the same dimension $k_0$ as $\CM$.
As $\CK_\Bn$ remains linear as $\Bn$ is varied one sees that $\CK_\Bn$ must  have the property that $W_{\Bm}[W_{\Bn}^{-1}(\CK_\Bn)]=\CK_\Bm$ for all unit vectors $\Bm$ and $\Bn$.
As $W_{\Bm}[W_{\Bn}^{-1}(\cdot)]$ is a nonlinear transformation, the image of a linear subspace under the transformation is generally a ``curved'' manifold,
so $\CK_\Bn$ has to be rather special for its image to be a linear subspace rather than a ``curved'' manifold. Through perturbation analysis with $\Bm$ close to $\Bn$,
Grabovsky established that $\CK_\Bn$ must be independent of $\Bn$, $\CK_\Bn=\CK_0$ for all $\Bn$, and he established that $\CK_0$ must satisfy 
the algebraic constraint that for every unit vector $\Bm\in\RR^d$ one has 
\beq   \BB\BGY(\Bm)\BB\in\CK_0,~~{\rm for~all~}\BB \in\CK_0, \eeq{0.a}
where the left-hand side of \eq{0.a} is to be regarded
as the composition of three linear maps each mapping $\CT$ to $\CT$ (or as the product of three $q\times q$ matrices if one takes a basis in $\CT$ and
represents each map by a matrix acting on the basis element),
and $\BGY(\Bm):\CT\rightarrow\CT$ depends on $\BL_0$ and the differential constraints on the fields relevant to
the physical problem under consideration. Explicitly, $\BGY(\Bm)$ is given by
\beq \BGY(\Bm)=\lang\BGG(\Bn)\rang_\Bn-\BGG(\Bm) \eeq{0.aaa}
where the angular brackets $\lang\cdot\rang_\Bn$ denote a possibly weighted average over the sphere $|\Bn|=1$ (for example, 
one could take a weighting concentrated at $\Bn=\Bn_0$,
giving $\lang\BGG(\Bn)\rang_\Bn=\BGG(\Bn_0)$) and $\BGG(\Bk)$ is given by
\beq \BGG(\Bk)=[\BGG_1(\Bk)\BL_0\BGG_1(\Bk)]^{-1}\BGG_1(\Bk), \eeq{twogam}
and the inverse is to be taken on the subspace $\CE_{\Bk}$ onto which $\BGG_1(\Bk)$ projects. 

\begin{figure}[!ht]
\centering
\includegraphics[width=0.8\textwidth]{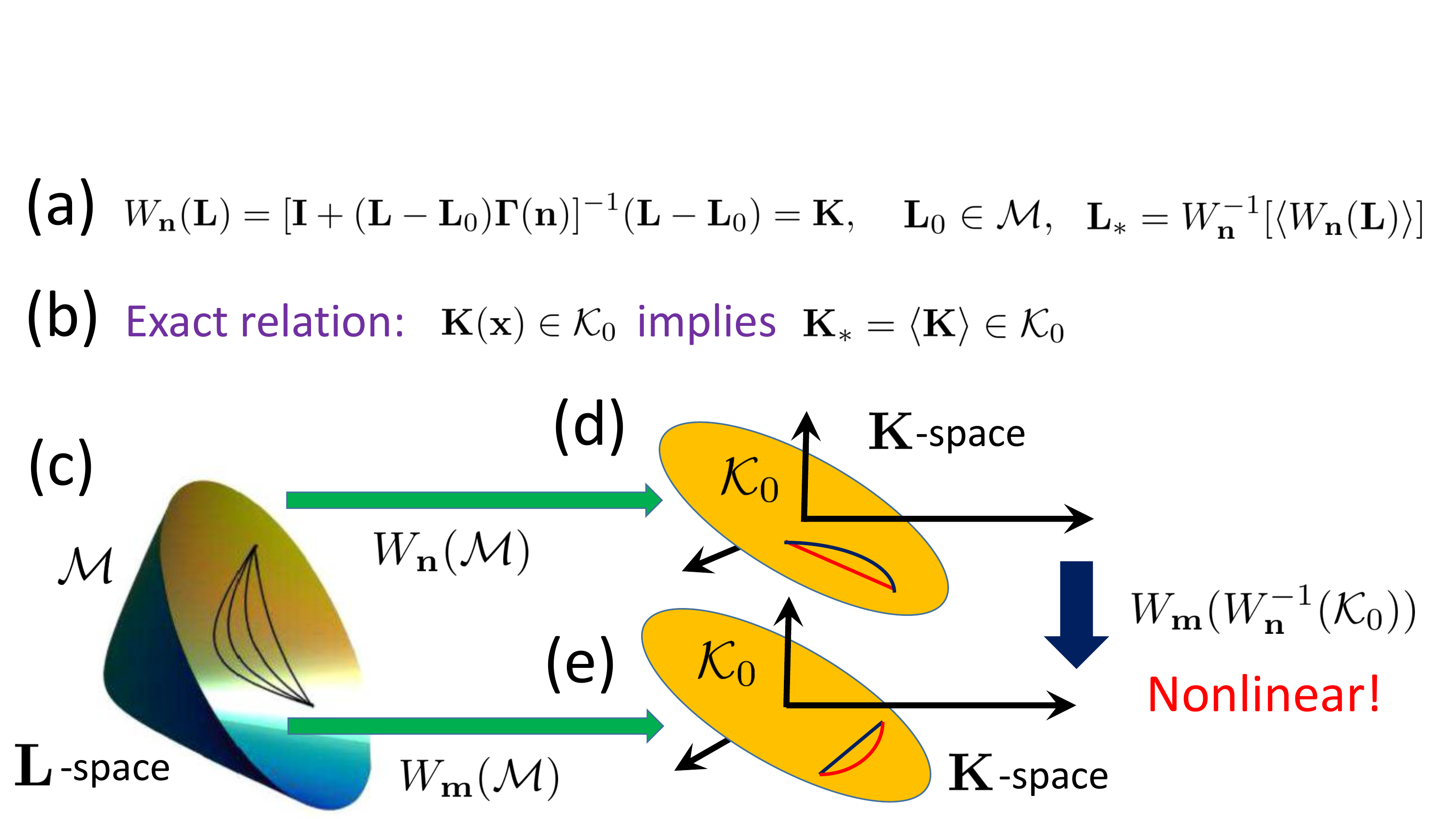}
\caption{\aj{Quick summary of the argument of Grabovsky \protect\cite{Grabovsky:1998:EREa}. The transformation $W_{\Bn}(\BL)$ defined in (a) reduces lamination in direction $\Bn$ to a linear average.
Therefore in $\BK$-space, for lamination in direction $\Bn$, one has (b), that $\BK_*=\langle\BK\rangle$, and so the set $\CK_0=W_{\Bn}(\CM)$ must be convex with no interior (as $\CM$ has no interior).
Thus $\CK_0$ must be a linear space (or a convex subset of a linear space). Choosing $\BL_0\in\CM$ guarantees that this linear space passes through the origin $0=W_{\Bn}(\BL_0)$ - thus $\CK_0$ is a
subspace, as in (c). It also must remain a subspace if we choose a different lamination direction $\Bm$. In figures (d) and (e) lamination in the directions $\Bn$ and $\Bm$ are represented by the red and blue
trajectories in the subspace $\CK_0$: these trajectories are straight lines on the surfaces $W_{\Bn}(\CM)$ and $W_{\Bm}(\CM)$ respectively. 
The subspace $\CK$ must be rather special in that $W_{\Bm}[W_{\Bn}^{-1}(\CK_0)]$ is a linear subspace,
even though the transformation $W_{\Bm}[W_{\Bn}^{-1}(\BK)]$ is a nonlinear transformation. This observation leads to the algebraic constraints on $\CK_0$ that are necessary and sufficient to ensure
that the set $\CM$ is stable under lamination.  \ay{The hyperbolic surface $\CM$ is reproduced from figure 4.1 of \protect\cite{Grabovsky:2016:CMM} (``\copyright IOP Publishing. Reproduced with permission. All rights reserved'')}}}
\labfig{2}
\end{figure}

\aj{As a simple example, for two dimensional conductivity with $\BL_0=\sigma_0\BI$ as our reference tensor, one sees that the $2\times 2$
matrix
\beq \BGY(\Bm)=(\BI-2\Bm\Bm^T)/2\sigma_0 \eeq{simex}
is trace-free and symmetric. We can then take $\CK_0$ as the subspace of trace-free and symmetric $2\times 2$ matrices. These have the property that the product of three of them (but not just two of them) is again 
trace-free and symmetric: assuming without loss of generality that one matrix is diagonal we have
\beq \begin{pmatrix} a & 0 \\ 0 & -a \end{pmatrix}
\begin{pmatrix} b & c \\ c & -b \end{pmatrix}
\begin{pmatrix} d & e \\ e & -d \end{pmatrix}=
\begin{pmatrix} abd+ace & & & abe-acd \\ abe-acd & & & -abd-ace \end{pmatrix},
\eeq{prod3}
which is again trace-free and symmetric, and thus the algebraic condition \eq{0.a} is satisfied.
Then the associated manifold ${\cal M}=W_{\bf n}^{-1}({\cal K}_0)$ consists of $2\times 2$ symmetric matrices with determinant $\sigma_0^2$,
and this is the manifold corresponding to the Dykhne \cite{Dykhne:1970:CTD} exact relation.}

Grabovsky's pioneering work, developed further with Sage in \cite{Grabovsky:1998:EREb},
provided essential clues that led to the breakthrough result \cite{Grabovsky:2000:ERE} establishing conditions 
that guarantee an exact relation holds for all composites, and not just laminates. Using carefully devised perturbation
expansions that had their basis in \cite{Milton:1990:RCF} Sect. 5,
coupled with analytic continuation arguments, one sees \cite{Grabovsky:2000:ERE}
that finding exact relations which hold for all composite geometries is tied with identifying tensor subspaces $\CK$ such that for all Fourier vectors $\Bk\ne 0$
one has
\beq \BB_1\BGY(\Bk)\BB_2\in\CK,~~{\rm for~all~}\BB_1,\BB_2\in\CK,
\eeq{0.1}
where $\BGY(\Bk)$ only depends on $\Bk/|\Bk|$, i.e. $\BGY(\Bk)=\BGY(\Bm)$ with $\Bm=\Bk/|\Bk|$ and $\BGY(\Bm)$ is the same operator as in \eq{0.aaa}. The space $\CK_0$ then could be $\CK$
or it may be just those symmetric or Hermitian matrices in $\CK$. (Previously in \cite{Milton:2002:TOC} Chap. 17, 
$\CK_0$ and $\CK$ had been labeled as $\CK$ and $\overline{\CK}$, respectively. We 
choose to drop the overline in $\overline{\CK}$ to simplify notation, as this space will be the focus of our analysis). Recently, Grabovsky \cite{Grabovsky:2017:MIF}
found a relation that holds for laminate geometries but not more general composites, so the condition \eq{0.a} is not sufficient
to guarantee that an exact relation holds for all microstructures: one needs to use \eq{0.1}.

The general theory of exact relations is very rich and Grabovsky and collaborators have systematically explored, and with tremendous effort,
exact relations for a wide variety of physically important problems, including conductivity with the Hall-effect, elasticity, piezoelectricity, thermoelasticity,
and thermoelectricity: \cite{Grabovsky:2016:CMM} gives a comprehensive review; see also \cite{Milton:2002:TOC} Chap. 17, and \cite{Grabovsky:2004:AGC}. To simplify the algebra
they assume $\CK$ has rotational invariance properties, so removing this assumption may yield a plethora of additional exact relations. The theory of exact relations
encompasses links between effective tensors: an example of such a link, for an isotropic 2-phase composite, is Levin's result \cite{Levin:1967:TEC} that the effective thermal 
expansion coefficient is known once the effective bulk modulus is measured.

Now one may ask: is there something deeper and more general behind these exact relations? Indeed, it is the purpose of this paper to reveal that there is something deeper. As indicated by
the argument presented in Fig.\fig{3}, exact relations should apply not only to effective tensors $\BL_*$ of periodic composites, but also to Dirichlet-to-Neumann maps of bodies containing
inhomogeneous media with inhomogeneities that are not necessarily small compared to the dimensions of $\GO$.
We formulate the problem slightly differently: in place of fields in $\CU$ are source terms, and we no longer require the subspaces $\CE$ and $\CJ$ to be comprised of periodic fields,
but rather fields that are square-integrable over $\RR^d$.

\begin{figure}[!ht]
\centering
\includegraphics[width=0.8\textwidth]{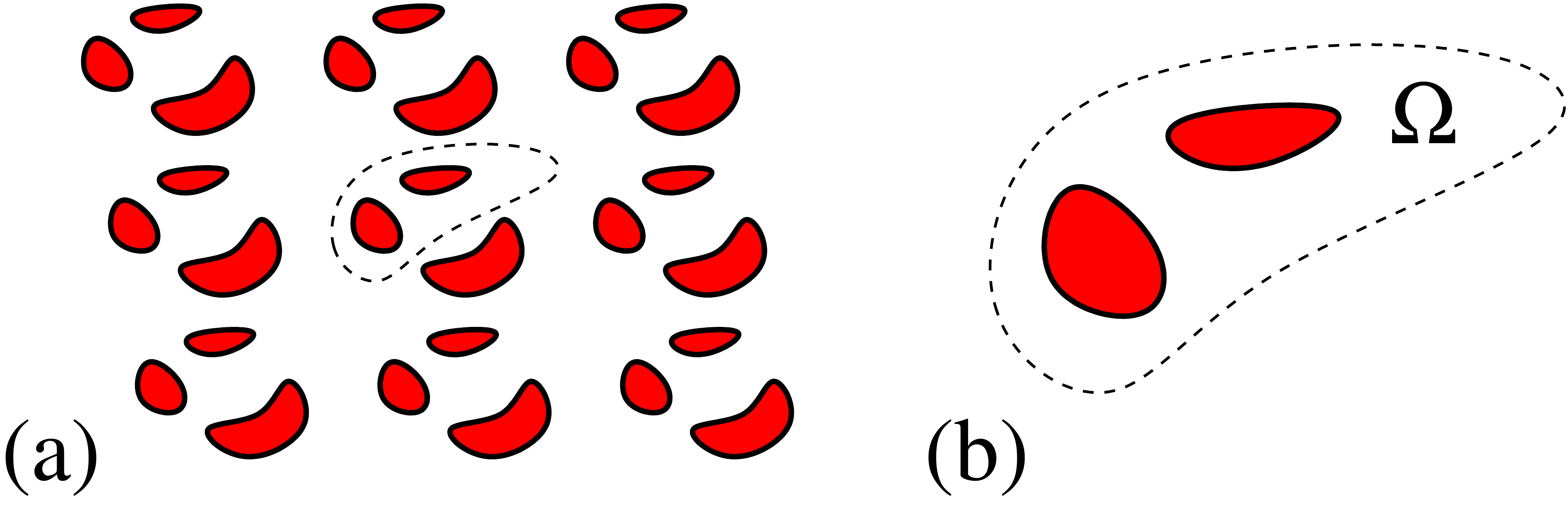}
\caption{\aj{The theory of exact relations for composites itself hints that there should be exact relations satisfied by the Dirichlet-to-Neumann map. Consider say a two-phase composite for
which an exact relation holds. The theory of exact relations for composites implies not only that $\BL_*\in\CM$ but also that suitably defined "polarization fields" $\mathbb{P}(\Bx)$ take
values in the subspace $\CK$ for all $\Bx$. Consider then the region $\GO$ outlined by the dashed lines in (a). It does not know that it is part of a periodic composite. Rather, the boundary
fields on $\Md\Omega$ must be such to force $\mathbb{P}(\Bx)$ within $\GO$ to take values in $\CK$. One of the goals of this paper is to identify these special boundary conditions }}
\labfig{3}
\end{figure} 

	\section{Functional framework}
\setcounter{equation}{0}
	\begin{Adef}
	\label{space}
	Let ${\cal H}={\cal J}\oplus{\cal E}$ be a tensor space of functions with $\RR^d$ as domain (e.g., $L^2(\RR^d)\otimes{\cal T}$, for a corresponding tensor space ${\cal T}$)
such that the projection $\BGG_1$ onto $\cal E$ acts locally in Fourier space, i.e., if $\BE=\BGG_1\BA$ then the Fourier components $\widehat{\BE}(\Bk)$ and $\widehat{\BA}(\Bk)$
of $\BE$ and $\BA$ are related via $\widehat{\BE}(\Bk)=\BGG_1(\Bk)\widehat{\BA}(\Bk)$ for some operator $\BGG_1(\Bk)$ that projects onto a subspace $\CE_{\Bk}\subset\CT$. \end{Adef}

In the case of the \aj{primary} equations \eq{b.0}, $\CE$ can be taken as the set of square-integrable fields $\BE(\Bx)$ such that $\BE=\Grad\Bu$ for some $m$ component potential $\Bu$, 
$\CJ$ can be taken as the set of square-integrable fields $\BJ(\Bx)$ such that $\Div\BJ=0$, $\CE_{\Bk}$ consists of rank-one $d\times m$ matrices of the form $\Bk\otimes\Ba$,
where $\Ba\in\mathbb{C}^m$, and consequently the action of $\BGG_1(\Bk)$ is given by
$\BGG_1(\Bk)\widehat{\BA}(\Bk)=\Bk\otimes(\Bk\cdot\widehat{\BA}(\Bk))/|\Bk|^2$ for $\Bk\ne 0$ and $\BGG_1(\Bk)\widehat{\BA}(\Bk)=0$ when $\Bk=0$. For
wider classes of partial differential equations, involving higher order derivatives, an explicit formula for $\BGG_1(\Bk)$ is given, for example, in
Section 12.2 of \cite{Milton:2002:TOC}, and in \cite{Milton:2013:SIG,Milton:2015:ATS}.

	Let $q$ denote the dimension of ${\cal T}$ and
	let ${\cal L}$ be the space of linear operators $\BA:{\cal H}\rightarrow{\cal H}$.
	Consider $\BL, \BL_0 \in{\cal L}$ defined as $\BL(\BR)(\Bx)=\BL(\Bx)\BR(\Bx)$ and respectively $\BL_0(\BR)(\Bx)=\BL_0\BR(\Bx)$,
         for $\BR\in {\cal H}$ 
        where $\BL(\Bx),\BL_0 \in L^{\infty}(\RR^d)\otimes L(\cal T)$ with $\BL_0$ denoting a given constant tensor. (Thus $\BL$ acts locally in space
while $\BL_0$ acts locally both in space and in Fourier space).
Assume that $\BL$ is self-adjoint, bounded and coercive, i.e., there exist constants $\Gb_0>\Ga_0>0$ such that
\beq \Gb_0\BI\geq \BL(\Bx) \geq \Ga_0\BI>0,\text{ for all }\Bx\in\RR^d, \eeq{boundcoer}
where the inequalities hold in the sense of the associated quadratic forms. We emphasize that many physical
problems where $\BL$ is not self-adjoint, including those where $\BL$ is complex and symmetric with a positive definite imaginary part,
can be converted to equivalent problems taking the required form \eq{I.1}, where the new $\BL$ is self-adjoint, bounded and coercive
(see, \cite{Cherkaev:1994:VPC},\cite{Milton:1990:CSP} Sect. 18, \cite{Cherkaev:2000:VMS} Chap. 13,\cite{Milton:2002:TOC} Sect. 12.11, \cite{Milton:2009:MVP}, \cite{Milton:2016:ETC} Sect. 5.2).
	We will next consider linear PDE's admitting a formulation in the following canonical form (see \cite{Milton:2016:ETC}),
	\begin{equation}
	\label{can-form}
	\BJ=\BL\BE-\Bh, \mbox{ with } \BJ\in {\cal J},\quad \BE\in {\cal E},\quad \Bh\in (\BL-\BL_0){\cal H}.
	\end{equation}
This is a restricted form of \eq{I.1} since in general $\BL-\BL_0$ may be singular and thus $(\BL-\BL_0){\cal H}$ does not equal ${\cal H}$.
The boundedness and coercivity conditions \eq{boundcoer} ensure that problem \eqref{can-form} 
has a unique solution for every $\Bh\in(\BL-\BL_0){\cal H}$.
        Since $\Bh\in(\BL-\BL_0){\cal H}$ we can equivalently let $\Bh=-(\BL-\BL_0)\Bs$ with $\Bs\in\CH$ and consider the equation
        \begin{equation}
	\label{can-alt}
	\BJ=\BL\BE+(\BL-\BL_0)\Bs, \mbox{ with } \BJ\in {\cal J},\quad \BE\in {\cal E},\quad \Bs\in {\cal H}.
	\end{equation}
	\begin{Adef}
	\label{Gamma-operator}
	Following \cite{Milton:2002:TOC} let us introduce the following  operator 
	$\BGG:{\cal H}\rightarrow {\cal H}$, defined by $\BGG \BA=\BE$ if and only if $\BE\in{\cal E}$ and $\BA-\BL_0\BE\in{\cal J}$. These equations
        are easily solved by going to Fourier space and one sees that $\BGG$ is univalued, with action in Fourier space given by the following lemma:
	\end{Adef}
 	\begin{lemma}
	\label{Gamma-op-prop}
	The operator $\BGG$ introduced in Definition \ref{Gamma-operator}, is self adjoint and acts locally in Fourier space. Explicitly, if $\BE=\BGG\BA$ then the 
         Fourier components $\widehat{\BE}(\Bk)$ and $\widehat{\BA}(\Bk)$ of $\BE$ and $\BA$ are related via $\widehat{\BE}(\Bk)=\BGG(\Bk)\widehat{\BA}(\Bk)$, where
$\BGG(\Bk)$ is defined by \eq{twogam}.
	\end{lemma}
To establish the lemma, suppose $\BE\in{\cal E}$ and $\BA-\BL_0\BE\in{\cal J}$. This, and the orthogonality of $\CE$ and $\CJ$, implies that for all $\Bk$ the Fourier components $\widehat{\BE}(\Bk)$ and $\widehat{\BA}(\Bk)-\BL_0\widehat{\BE}(\Bk)$ lie in $\CE_{\Bk}$ and its orthogonal complement, respectively. Recalling that $\BGG_1(\Bk)$ projects onto $\CE_{\Bk}$ we obtain
\beq 0=\BGG_1(\Bk)[\widehat{\BA}(\Bk)-\BL_0\widehat{\BE}(\Bk)]=\BGG_1(\Bk)\widehat{\BA}(\Bk)-[\BGG_1(\Bk)\BL_0\BGG_1(\Bk)]\BGG_1(\Bk)\widehat{\BE}(\Bk), \eeq{lempr}
and this is easily solved for $\widehat{\BE}(\Bk)$, yielding $\widehat{\BE}(\Bk)=\BGG(\Bk)\widehat{\BA}(\Bk)$, where $\BGG(\Bk)$ is defined by \eq{twogam}. The expression
for $\BGG(\Bk)$ can equivalently be rewritten as
\beq \BGG(\Bk)=\BGG_1(\Bk)[\BGG_1(\Bk)\BL_0\BGG_1(\Bk)]^{-1}\BGG_1(\Bk), \eeq{twogamalt}
\aj{(where the inverse is to taken on the space $\CE_{\Bk}$)} which is evidently self-adjoint, as $\BGG_1(\Bk)$ and $\BL_0$ are self adjoint. \ayy{Note also that $\BL_0\BGG(\Bk)$ can also be interpreted as a non-orthogonal projection onto $\BL_0{\cal E}$ along ${\cal J}$.}


\ayy{Let $\BM:\CT\to\CT$ be a self adjoint positive semidefinite operator. Using the definitions of $\BL\in {\cal L}$ and $\BL_0 \in {\cal L}$ we define the operator $\BK:{\cal H}\rightarrow {\cal H}$ as $\BK(\BR)(\Bx)=\BK(\Bx)\BR(\Bx)$ with 
\begin{equation}
\label{K(x)}
\BK(\Bx)=W_{\BM}(\BL(\Bx))=
[\BI+(\BL(\Bx)-\BL_0)\BM]^{-1}(\BL(\Bx)-\BL_0)\in L(\cal T). \end{equation}
\aj{It follows from Grabovsky's definition 3.17 and lemma 3.18 \cite{Grabovsky:2016:CMM}
	that $\BI+(\BL(\Bx)-\BL_0)\BM$ is invertible when $\BM\BL_0\BM\leq \BM$, and so under this assumption
	the fractional linear transformation $W_{\BM}:L(\cal T)\to L(\cal T)$ given by \eqref{K(x)} is well defined.}}

	For $\lambda\in[0,1]$ consider the sequence of operators $\BL_\lambda:{\cal H}\rightarrow {\cal H}$ defined by
	\begin{equation}  
	\label{L-lambda}
	\BL_\lambda=\BL_0+\lambda[\BI+(1-\lambda)(\BL-\BL_0)\BM]^{-1}(\BL-\BL_0),
	\end{equation}
	and note that $\BL_{\lambda=1}\!=\!\BL$ and $\BL_{\lambda=0}\!=\!\BL_0$. 
\ay{We also point out that $\BL_\lambda$ is just a homothety by $\lambda$ in the $W$ variables, i.e., $W_{\BM}(\BL_\lambda)=\lambda W_{\BM}(\BL)$. 
In particular, in the context of composites $\BL_\lambda$ with $\BM=\BGG(\Bn)$ is in fact the effective tensor of a laminate in direction $\Bn$ 
of $\BL$ and $\BL_0$ with a volume fraction $\lambda$ of $\BL$.}
	
	Here we assume these operators $\BL_\lambda$ are well-defined
(which is a consequence of Theorem 4.2 under some restrictions on $\BL$ and $\BL_0$).
	Let $\BGY\in{\cal L}$, be defined by $\BGY=\BM-\BGG$. As $\BGG$ acts locally in Fourier space,
if $\BB=\BGY\BA$, then the Fourier components $\widehat{\BB}(\Bk)$ and $\widehat{\BA}(\Bk)$ of $\BB$ and $\BA$ satisfy a local relation
$\widehat{\BB}(\Bk)=\BGY(\Bk)\widehat{\BA}(\Bk)$ with $\BGY(\Bk)=\BM-\BGG(\Bk)$ taking values in $L(\cal T)$.
Assume there exists a subspace ${\cal K}\subset L(\cal T)$ such that for all $\Bk$
	\begin{equation}
	\label{K} {\cal K}\BGY(\Bk){\cal K}\doteq\{\BB_1\BGY(\Bk)\BB_2, \mbox{ for } \BB_1, \BB_2\in{\cal K}\}\subset {\cal K},
	\end{equation} 
in which $\BB_1\BGY(\Bk)\BB_2$ is the composition of the three maps $\BB_2:\CT\to\CT$, $\BGY(\Bk):\CT\to\CT$, and $\BB_1:\CT\to\CT$. If
\eqref{K} holds for all $\Bk$ then it clearly holds if $\BGY(\Bk)$ is replaced by any tensor $\BA$ in the subspace $\CA$ spanned by the
$\BGY(\Bk)$ as $\Bk$ varies. Hence \eqref{K} can be rewritten as
\beq \CK\CA\CK\doteq\{\BB_1\BA\BB_2, \mbox{ for } \BB_1, \BB_2\in{\cal K},\,\BA\in\CA\}\subset\CK. \eeq{Kalt}
Spaces $\CK$ having this property have been called an associative $\CA$-multialgebra by Grabovsky \cite{Grabovsky:2016:CMM}. Instead
of testing that \eq{K} holds for all $\BGY(\Bk)$ as $\Bk$ varies, it suffices to test it for a basis of $\CA$. 
	
	Next, let us denote by $\Be_1,\Be_2,...,\Be_q$ a basis of ${\cal T}$. 
For given $q$ functions $\Bs_1,\Bs_2,...,\Bs_q\in{\cal H}$ consider the following linear map $\mathbb{S}:{\cal T}\rightarrow {\cal H}$ defined by 
	\begin{equation}
	\label{G} 
	\mathbb{S}\Be_i=\Bs_i, \mbox{ for all } i\in\{1,...,q\}.
	\end{equation} 
 We let $\mathbb{S}(\Bx)$ denote the associated field taking for each $\Bx$ values in $L(\cal T)$ such that $\mathbb{S}(\Bx)\Be_i=\Bs_i(\Bx)$ for all $i$.
This field can be considered to lie in the space $\mathfrak{H}=L^{2}(\RR^d)\otimes L(\cal T)$, endowed with the inner product
	\begin{equation}
	\label{ip}
         \langle \mathbb{A}, \mathbb{B} \rangle_{\mathfrak{H}}=\sum_{i=1}^q \langle\mathbb{A}\Be_i,\mathbb{B}\Be_i\rangle_{\cal H}.
	\end{equation} 	
	Note that any linear operator $\BF: \CH \rightarrow \CH $, such as $\BK$ or $\BGY$, has a natural extension to an operator on $\mathfrak{H}$: we define
	\begin{equation}
	\label{oext}
        \BF\mathbb{A}=\mathbb{B} \mbox{ iff } \BF(\mathbb{A}\Be_i)=\mathbb{B}\Be_i \mbox{ for all }i,
	\end{equation} 
        where, to simplify notation, we use the same symbol for the operator acting on $\mathfrak{H}$ as for the operator acting on $\CH$. 

	\section{The central theorem}
\setcounter{equation}{0}
	Define ${\cal S}$ as a subspace of $L(\cal T)$ such that $\BA\cal S \subset \cal K$ for all $\BA\in\cal K$. For example, \eqref{K}
implies $\cal S$ could be taken as ${\cal Q}=\CA\CK$, defined as the space spanned by $\BGY(\Bk)\CK$ as $\Bk$ varies. \ay{A natural choice for ${\cal S}$ 
is the largest subspace with the property that $\BA\cal S \subset \cal K$, although it is then not clear how easily that subspace can be computed.} 
        The central theorem of this paper states:
	
	\begin{Thm}
	\label{teorema-principala}
	Consider problem \eqref{can-form} and let $\BL_0,\BM,\BL$ satisfy all the conditions presented in the previous section. Assume that the following conditions hold:
	\begin{eqnarray}
	&& \BK(\Bx)=W_\BM(\BL(\Bx))\in {\cal K}\quad \mbox{ for all }\, \Bx,\label{H1}\\
	&& \BL_\lambda \quad \mbox{is bounded and coercive on } {\cal H} \mbox{ for all }\lambda\in [0,1],\label{H2}\\
	&& \Bs_1,\Bs_2,...,\Bs_q\in{\cal H}\mbox{ are such that } \mathbb{S}(\Bx) \in \CS \mbox{ for all }\,\Bx. \label{H3}
	\end{eqnarray}
	{{where $\BL_\lambda$ was defined at \eqref{L-lambda} and $W_\BM$ was defined at \eqref{K(x)}. Next, consider the set of $q$ sources $\Bh_i=(\BL-\BL_0)\Bs_i$,
			$i=1,2,\ldots,q$, and let  $(\BE_i,\BJ_i)$ ($i=\{1,..,q\}$) denote the unique solution of the problem \eqref{can-form} for each of the sources $\Bh_1,\Bh_2,...,\Bh_q$ respectively. For each solution pair $(\BE_i,\BJ_i)$ define the corresponding polarization field via $\BP_i=\BJ_i-\BL_0\BE_i$ and introduce the operator $\mathbb{P}:{\cal T}\rightarrow {\cal H}$ defined by 	
\begin{equation}
			\label{gen-P} \mathbb{P}\Be_i=\BP_i, \mbox{ for all }\, i\in\{1,...,q\}.
			\end{equation}
Associated with $\mathbb{P}$ is the field $\mathbb{P}(\Bx)$ taking values for each $\Bx$ in $L(\cal T)$ such that $\mathbb{P}(\Bx)\Be_i=\BP_i(\Bx)$ for all $i$.}} Then 
\beq \mathbb{P}(\Bx)\in \CK  \mbox{ for all }\, \Bx. \eeq{mr}
	\end{Thm}
	
	\begin{proof}

        Our proof of this result has much in common with the proof establishing sufficient conditions for an exact relation to hold for all composite geometries
        (see \cite{Grabovsky:2000:ERE}, \cite{Milton:2002:TOC} Sect.  17.3, \cite{Grabovsky:2016:CMM} Sect. 4.5, and \cite{Grabovsky:2017:MIF}).
        From \eqref{H2} we have that, for any given $\Bh\in(\BL-\BL_0){\cal H}$ there exists 
        unique $\BE_\lambda\in {\cal E}$ and $\BJ_\lambda\in{\cal J}$ that solve \eqref{can-form} with $\BL=\BL_\Gl$. We choose to define the 
        polarization field $\BP_\lambda$ as,
	\begin{equation}
	\label{polarization}
	\BP_\lambda=\BJ_\lambda-\BL_0\BE_\lambda =(\BL_\lambda-\BL_0)\BE_\lambda-\Bh.
	\end{equation}
	(Note that the polarization field $\BP_\lambda$ is not $(\BL_\lambda-\BL_0)\BE_\lambda$.) Then
        from Definition \ref{Gamma-operator} we have that 
	\begin{equation}
	\label{Gamma-P}
	\BGG\BP_\lambda=\BGG(\BJ_\lambda-\BL_0\BE_\lambda)=-\BE_\lambda.
	\end{equation}
	Indeed, \eqref{Gamma-P} follows from $-\BE_\lambda\in{\cal E}$ and $\BJ_\lambda-\BL_0\BE_\lambda-\BL_0(-\BE_\lambda)=\BJ_\lambda\in{\cal J}$. Hence we obtain
	\begin{eqnarray}
	\label{1}
	[\BI+(\BL_\lambda-\BL_0)\BGG]\BP_\lambda & = & \BP_\lambda-(\BL_\lambda-\BL_0)\BE_\lambda\nonumber\\
	& = & (\BL_\lambda-\BL_0)\BE_\lambda-\Bh-(\BL_\lambda-\BL_0)\BE_\lambda\nonumber\\
	& = & -\Bh.
	\end{eqnarray}
	Thus from \eqref{1} together with the uniqueness of $\BE_\Gl,\BJ_\Gl$ we have that, for all $\lambda\in[0,1]$,
	\begin{equation}
	\label{inv-L-lambda}
	[\BI+(\BL_\lambda-\BL_0)\BGG]^{-1}: (\BL-\BL_0)\cal{H}\rightarrow \cal{H}\mbox{ is a well-defined linear operator}.
	\end{equation}
\aj{With $\Gl=1$, this result is Grabovsky's corollary 3.19 \cite{Grabovsky:2016:CMM} but this follows in our case from different arguments than in his book.}
	Next, for $\lambda=1$ in \eqref{1} we obtain
  \begin{equation}
  \label{2}
  \BP=-[\BI+(\BL-\BL_0)\BGG]^{-1}\Bh,
  \end{equation}
where in \eqref{2} and in what follows we use $\BP$ instead of $\BP_{\lambda=1}$.
This may be equivalently rewritten as follows (see \cite{Grabovsky:2000:ERE} Sect 3.2, or \cite{Milton:2002:TOC} Sect. 14.9, for a similar approach)
  \begin{eqnarray}
  \label{3}
\BP & = & -[\BI+(\BL-\BL_0)\BGG]^{-1}\Bh\nonumber\\
&=&-[\BI+(\BL-\BL_0)\BM+(\BL-\BL_0)(\BGG-\BM)]^{-1}\Bh\nonumber\\
&=&[\BI-[\BI+(\BL-\BL_0)\BM]^{-1}(\BL-\BL_0)(\BM-\BGG)]^{-1}[\BI+(\BL-\BL_0)\BM]^{-1}(\BL-\BL_0)\Bs \nonumber\\
&=&(\BI-\BK\BGY)^{-1}\BK \Bs,
  \end{eqnarray}
  where $\Bh=-(\BL-\BL_0)\Bs$, $\BK\in{\cal L}$ was introduced at \eqref{K(x)}, and $\BGY=\BM-\BGG$ was defined at \eqref{K}.
 Using the notation introduced immediately after \eqref{gen-P}, equality \eqref{3} 
can be equivalently written as
  \begin{equation}
  \label{4}
  \mathbb{P}=(\BI-\BK\BGY)^{-1}\BK \mathbb{S},
  \end{equation}
where $\mathbb{S}:{\cal T}\rightarrow {\cal H}$ was defined at \eqref{G}. We choose to regard $\mathbb{S}$ and $\mathbb{P}$ as fields in $\mathfrak{H}$
and regard $\BK$ and $\BGY$ in \eqref{4} as operators acting in $\mathfrak{H}$, defined according to \eqref{oext}. Similarly we can define fields $\mathbb{E}(\Bx)$
and $\mathbb{J}(\Bx)$ via
\beq \mathbb{E}(\Bx)\Be_i=\BE_i(\Bx),\quad \mathbb{J}(\Bx)\Be_i=\BE_i(\Bx), 
\eeq{extf}
and the governing equations become
\beq \mathbb{J}(\Bx)=\BL(\Bx)\mathbb{E}(\Bx)+(\BL(\Bx)-\BL_0)\mathbb{S}(\Bx),\mbox{ with }\mathbb{E}\in\mathfrak{E},\quad \mathbb{J}\in\mathfrak{J},
\eeq{extgov}
in which $\mathfrak{E}$ is comprised of fields $\mathbb{E}$ such that $\mathbb{E}\Be_i\in\CE$ for all $i$,
while $\mathfrak{J}$ is comprised of fields $\mathbb{J}$ such that $\mathbb{J}\Be_i\in\CJ$ for all $i$.

Consider the following sequence of related fields,
  \begin{equation}
    \label{5}
    \mathbb{P}_{\lambda}=(\BI-\lambda \BK\BGY)^{-1}\lambda\BK\mathbb{S}.
    \end{equation}
  For small $\lambda$ the Neumann series for $\lambda \BK\BGY$ is convergent and we have 
  \begin{eqnarray}
      \label{6}
      \mathbb{P}_{\lambda}&=&(\BI-\lambda \BK\BGY)^{-1}\lambda\BK\mathbb{S}\nonumber\\
      &=& \sum_{j=0}^{\infty}\lambda^{j+1} (\BK\BGY)^{j}\BK\mathbb{S}.
      \end{eqnarray}
It is to be emphasized that in this expansion $\BK$ and $\BGY$ are operators: they act on the field in $\mathfrak{H}$ to the right of them. Related expansions
in the theory of composites were first introduced in \cite{Milton:1990:RCF}, sect.5, for the conductivity problem, and their convergence properties,
allowing for possibly non-symmetric conductivity tensors, were studied in \cite{Clark:1994:MEC}. They also form the basis of accelerated iterative Fast
Fourier transform (FFT) techniques for evaluating the fields in composites and the associated effective tensors \cite{Eyre:1999:FNS} 
(see also \cite{Milton:2002:TOC} Sect. 14.9 and Sect. 14.10), that generally
converge faster than the iterative FFT techniques first proposed in \cite{Moulinec:1994:FNM}. However the application that motivates their introduction
in our paper is their essential role in the theory of exact relations in composites \cite{Grabovsky:2000:ERE}. 

To prove $\mathbb{P}_{\lambda}(\Bx)$ takes values in $\CK$ when $\mathbb{S}(\Bx)$ takes values in $\CS$, one proceeds by induction. Define the partial sums
  \begin{equation}
      \label{6.5}
      \mathbb{P}_{\lambda}^m=\sum_{j=0}^{m}\lambda^{j+1} (\BK\BGY)^{j}\BK\mathbb{S}, \quad \mathbb{Q}_{\lambda}^m=\BGY\sum_{j=0}^{m}\lambda^{j+1} (\BK\BGY)^{j}\BK\mathbb{S}.
 \end{equation}
Clearly these fields, which are in are $\mathfrak{H}$, are related by
  \begin{equation}
      \label{6.6}
      \mathbb{P}_{\lambda}^{m+1}=\Gl\BK\mathbb{Q}_{\lambda}^m+\Gl\BK\mathbb{S},\quad  \mathbb{Q}_{\lambda}^m=\BGY\mathbb{P}_{\lambda}^m.
 \end{equation}
Assume for some $m$ that for every $\Bx$, $\mathbb{P}_{\lambda}^m(\Bx)\in \CK$. This is clearly true when $m=0$ by the definition of $\CS$.
Then the Fourier components of $\mathbb{P}_{\lambda}^m$ also lie in $\CK$. It follows that the Fourier components of
$\mathbb{Q}_{\lambda}^m$, and hence also the values of $\mathbb{Q}_{\lambda}^m(\Bx)$, lie in $\CQ$, defined (as in the beginning of this section) as
the space spanned by $\BGY(\Bk)\CK$ as $\Bk$ varies. Then since $\BA\CQ\subset\CK$
for all $\BA\in\CK$ we deduce that $\mathbb{P}_{\lambda}^{m+1}(\Bx)$ lies in $\CK$ for all $\Bx$.  

{{We conclude that there exists $\lambda_0>0$ such that for $\lambda<\lambda_0$}}, $\mathbb{P}_{\lambda}(\Bx)$ defined by \eqref{5} lies in $\CK$ for all $\Bx$. This implies 
that for small enough $\lambda$ we have
\begin{equation}
\label{8}
f_{\mathbb{F}}(\Gl)=\langle (\BI-\lambda \BK\BGY)^{-1}\lambda\BK \mathbb{S},\mathbb{F}\rangle_{\mathfrak{H}}=0, 
\mbox{ for all } \mathbb{F}\in\mathfrak{H} \mbox{ with } \mathbb{F}(\Bx) \in \CK^{\perp} \mbox{ for all } \Bx,
\end{equation}
in which $\CK^{\perp}$ is the orthogonal complement of $\CK$ in the space $L(\cal T)$ with respect to the inner product defined (analogously to \eqref{ip}) by
	\begin{equation}
	\label{20a}
         \langle \BA, \BB \rangle_{L(\cal T)}=\sum_{i=1}^q \langle\BA\Be_i,\BB\Be_i\rangle_{\CT} \text{ for all }\BA,\BB\in L(\cal T).
	\end{equation} 	

Next we will prove that $(\BI-\lambda \BK\BGY)^{-1}\lambda\BK \mathbb{S}$ is analytic on an open set $D\subset {\mathbb C}$ with $[\displaystyle\frac{\lambda_0}{2},1]\subset D$.
Indeed note that the operator function $(\BI-\lambda \BK\BGY)^{-1}:{\mathbb C}\rightarrow {\cal L}$ is analytic for $\lambda$ such that $\displaystyle\frac{1}{\lambda}\in\rho(\BK\BGY)$  (where $\rho(\BK\BGY)$ denotes the resolvent set of $\BK\BGY$) and therefore the function $(\BI-\lambda \BK\BGY)^{-1}\lambda\BK \mathbb{S}$ will be analytic on this set of $\lambda$ values (see \cite{Lax:2002:FA} Chap. 17). 
   
Thus, we observe that using the openness of $\rho(\BK\BGY)$ it is enough to show that $[1,\displaystyle\frac{2}{\lambda_0}]\!\subset\! \rho(\BK\BGY)$, as this will imply that there exists an open set $D_1$ (bounded above and below by positive numbers) with $[1,\displaystyle\frac{2}{\lambda_0}]\subset D_1 \subset\! \rho(\BK\BGY)$ and in turn this will give the existence of an open set $D\subset \left\{\lambda, \displaystyle\frac{1}{\lambda}\in D_1\right\}$ with $[\displaystyle\frac{\lambda_0}{2},1]\subset D$  such that $(\BI-\lambda \BK\BGY)^{-1}\BK \mathbb{S}$ is analytic on $D$.
    
    We have that
    \begin{eqnarray}
    \label{9}
    (\BI-\lambda \BK\BGY)^{-1}\lambda\BK & = &[\BI-\lambda[\BI+(\BL-\BL_0)\BM]^{-1}(\BL-\BL_0)\BGY]^{-1}\lambda\BK\nonumber\\
    &=&[\BI+(\BL-\BL_0)\BM-\lambda(\BL-\BL_0)(\BM-\BGG)]^{-1}[\BI+(\BL-\BL_0)\BM]\lambda\BK\nonumber\\
    &=&\lambda[\BI+(1-\Gl)(\BL-\BL_0)\BM+\lambda(\BL-\BL_0)\BGG]^{-1}(\BL-\BL_0) \nonumber\\
    &=&\lambda[\BI+\Gl[\BI+(1-\Gl)(\BL-\BL_0)\BM]^{-1}(\BL-\BL_0)\BGG]^{-1}[\BI+(1-\Gl)(\BL-\BL_0)\BM]^{-1}(\BL-\BL_0)\nonumber\\
    &=&\left[\BI+(\BL_\lambda-\BL_0)\BGG\right]^{-1}(\BL_\lambda-\BL_0), {{\mbox{ for } \lambda\in(0,1]}},\nonumber\\
    \end{eqnarray}
    Then, using the fact that by definition, for $\Gl\ne 0$, we have
\beq \BL_\lambda-\BL_0= \lambda[\BI+(1-\lambda)(\BL-\BL_0)\BM]^{-1}(\BL-\BL_0), \eeq{ad.1}
or equivalently,
\beq (\BL_\lambda-\BL_0)=\Gl(\BL-\BL_0)-
                                      (1-\lambda)(\BL-\BL_0)\BM(\BL_\lambda-\BL_0),
\eeq{ad.2}
which implies
\beq \BL_\lambda-\BL_0=(\BL-\BL_0)[\lambda \BI-(1-\lambda)\BM(\BL_\lambda-\BL_0)].
\eeq{ad.3}
We see that $(\BL_\lambda-\BL_0)\cal{H}\subset$$(\BL-\BL_{0})\cal{H}$ and this together
with \eqref{inv-L-lambda} gives
    \begin{equation}
    \label{10}
    (\BI-\lambda \BK\BGY)^{-1}\lambda\BK\in{\cal L} \mbox{ for all }\lambda\in[0,1],
    \end{equation}
and this implies 
    $$(\BI-\lambda \BK\BGY)^{-1}\lambda\BK\BGY\in{\cal L} \Rightarrow -\BI +(\BI-\lambda \BK\BGY)^{-1}\in{\cal L}\Rightarrow (\BI-\lambda \BK\BGY)^{-1}\in{\cal L} \mbox{ for } \lambda\in(0,1].$$
    The last result above implies that $[1,\infty)\!\subset\! \rho(\BK\BGY)$ which in turn as explained above implies that there existence an open set $D$ with $[\displaystyle\frac{\lambda_0}{2},1]\subset D$  such that $(\BI-\lambda \BK\BGY)^{-1}\BK \mathbb{S}$ is analytic on $D$.

   This together with \eqref{8} and by analytic continuation in the complex plane implies that 
    
    \begin{equation}
    \label{11}
    f_{\mathbb{F}}(1)=\langle (\BI-\BK\BGY)^{-1}\BK \mathbb{S},\mathbb{F}\rangle_{\mathfrak{H}}=0, 
 \mbox{ for all } \mathbb{F}\in\mathfrak{H} \mbox{ with } \mathbb{F}(\Bx) \in \CK^{\perp} \mbox{ for all } \Bx.
     \end{equation}
    
    Thus, as desired, we conclude that $\mathbb{P}=(\BI-\BK\BGY)^{-1}\BK \mathbb{S}$ takes values in $\CK$ when $\mathbb{S}(\Bx)\in\CS$ for all $\Bx$.
	\end{proof}
	
The fundamental algebraic property $\eqref{K}$ is clearly an algebraic property of the subspace $\CK$ comprised
of operators $\BK$ mapping $\CT$ to $\CT$ and is independent of what basis $\Be_i$, $i=1,2,3,\ldots,q$, for $\CT$
we may choose. Up to now we could apply our theory when the field $\mathbb{S}(\Bx)$
took values in $\CS$. Suppose instead that, for some nonsingular mapping $\BD:\CT\to\CT$,
$\mathbb{S}(\Bx)\BD^{-1}$, regarded as the composition of the two maps $\BD^{-1}:\CT\to\CT$ and
$\mathbb{S}(\Bx):\CT\to\CT$, took values in $\CS$. In this case we can introduce a new basis
\beq \Be'_i=\BD\Be_i,\quad i=1,2,3,\ldots,q,
\eeq{e.1}
and define $\mathbb{S}'(\Bx)$ as that field taking values in $L(\cal T)$ such that
\beq \mathbb{S}'(\Bx)\Be'_i=\Bs_i=\mathbb{S}(\Bx)\Be_i=\mathbb{S}(\Bx)\BD^{-1}\Be'_i,
\eeq{e.2}
implying $\mathbb{S}'(\Bx)=\mathbb{S}(\Bx)\BD^{-1}$. Accordingly, we need to introduce the field $\mathbb{P}'(\Bx)$ 
as that field taking values in $L(\cal T)$ such that
\beq \mathbb{P}'(\Bx)\Be'_i=\BP_i=\mathbb{P}(\Bx)\Be_i=\mathbb{P}(\Bx)\BD^{-1}\Be'_i,
\eeq{e.3}
implying $\mathbb{P}'(\Bx)=\mathbb{P}(\Bx)\BD^{-1}$. Our theorem says that $\mathbb{P}'(\Bx)$ takes values in
$\CK$ when $\mathbb{S}'(\Bx)$ takes values in $\CS$, and so we conclude that, for all nonsingular $\BD:\CT\to\CT$ and for any $\Bx\in\mathbb{R}^d$,
\beq \mathbb{P}(\Bx)\BD^{-1}\in\CK~~\text{when}~~\mathbb{S}(\Bx')\BD^{-1}\in\CS~~\text{for all}~~\Bx',
\eeq{e.4}
or equivalently that
\beq \mathbb{P}(\Bx)\in\CK\BD~~\text{when}~~\mathbb{S}(\Bx')\in\CS\BD~~\text{for all}~~\Bx',
\eeq{e.4a}
where 
\beqa  \CS\BD=\{\BA\in L(\cal T)~|~\BA=\BS\BD~~\text{for some}~\BS\in\CS\}, \nonum
 \CK\BD=\{\BA\in L(\cal T)~|~\BA=\BK\BD~~\text{for some}~\BK\in\CK\}.
\eeqa{e.00}

It is not immediately clear when the assumption of the central theorem that $\BL_\lambda$ is bounded and coercive for all $\lambda\in[0,1]$ is satisfied. The following theorem gives a simple condition on $\BM$ that guarantees this.

\ayy{
\begin{Thm}
	\label{example}
	Let $\BM$ be such that $\BM\BL_0\BM\leq\BM$, and assume that $\BL,\BL_0$ are coercive. Then $\BL_\lambda$ defined at \eqref{L-lambda} is bounded and coercive for all $\lambda\in[0,1]$.
\end{Thm}
}

\begin{proof}
	\ayy{
	The boundedness of $\BL_\lambda$ can be seen to be a corollary of Grabovsky's Definition 3.17 and Lemma 3.18 \cite{Grabovsky:2016:CMM}. Then we remark that, condition $\BM\BL_0\BM\leq\BM$ implies the fact that the family of self-adjoint operators $\BL_\lambda$ is continuous in $\lambda\in[0,1]$. Therefore, eigenvalues of $\BL_\lambda$ depend continuously on $\lambda$. }

	\ayy{When $\lambda=1$, or $\lambda=0$, we have $\BL_\lambda=\BL$ and respectively $\BL_\lambda=\BL_0$ which are bounded and coercive by hypothesis. }
	
\ayy{	Thus, if there exists a $\lambda$ for which $\BL_\lambda$ is not coercive, then there exists $\lambda'\in (0,1)$  so that at least one of the eigenvalues of $\BL_\lambda'$ is zero, which in turn will imply that $\BL_\lambda'$ is a singular matrix. Hence, to prove coercivity for the family $\BL_\lambda$ for all $\lambda\in (0,1)$ it is sufficient to show that $\BL_\lambda$ is invertible for all $\lambda\in (0,1)$. }
	
\ayy{	In this regard, if we consider 
	\begin{equation}
	\label{not1} \BL'_\lambda=\BL^{-1/2}_0\BL_\lambda\BL^{-1/2}_0,\; \BL'=\BL^{-1/2}_0\BL\BL^{-1/2}_0,\; \BM'=\BL^{1/2}_0\BM\BL^{1/2}_0
	\end{equation}
	\eqref{L-lambda} becomes
	\begin{equation}
	\label{not2}
	\BL'_\lambda=\BI+\lambda[\BI+(1-\lambda)(\BL'-\BI)\BM']^{-1}(\BL'-\BI),
	\end{equation}
	where the condition $\BM\BL_0\BM\leq\BM$ is rewritten for $\BM'$ as $(\BM')^2\leq \BM'$ and we also have $\BL'>0$. Then invertibility of $\BL_\lambda$ is equivalent to invertibility of $\BL'_\lambda$ which in turn is equivalent to the invertibility of
	\begin{equation}
	\label{not3}
	\BF=[\BI+(1-\lambda)(\BL'-\BI)\BM']\BL'_\lambda=(1-\lambda)(\BI-\BM')+\BL'(\lambda\BI+(1-\lambda)\BM').
	\end{equation}
	The eigenvalues of $\BM'$ are between 0 and 1 and therefore the eigenvalues of $\lambda\BI+(1-\lambda)\BM'$ will always be between $\lambda$ and 1. Thus invertibility of $\BF$ is equivalent to the invertibility of 
	\begin{equation}
	\label{not4}
	\BF'=\BD+\BL', \;\BD=(1-\lambda)(\BI-\BM')(\lambda\BI+(1-\lambda)\BM')^{-1}.
	\end{equation}
	The eigenvalues of the self-adjoint operator $\BD$ are clearly non-negative and therefore the operator $\BD+\BL'$ is positive definite and hence, invertible.
}
\end{proof}

	\begin{Arem}
	Here we show that the condition \eqref{K} simplifies in the case where for $\Bk\ne 0$, $\BGG(\Bk)$ only depends on $\Bk/|\Bk|$,
\ay{and that $\BM$ can be eliminated from the condition}.
 By subtracting the conditions implied by \eqref{K} that
\beq {\cal K}(\BM-\BGG(\Bn)){\cal K}\subset \CK,\quad {\cal K}(\BM-\BGG(\Bm)){\cal K}\subset \CK, \eeq{z.1}
we get that ${\cal K}(\BGG(\Bn)-\BGG(\Bm)){\cal K}\subset \CK$ for all unit vectors $\Bm$ and $\Bn$. Hence if \eq{z.1} holds it will still
hold if $\BM$ is replaced by $\lang\BGG(\Bn)\rang_{\Bn}$ where the angular brackets $\lang\cdot\rang_\Bn$ denote a possibly weighted average over the sphere $|\Bn|=1$.
The formula \eq{twogam} for $\BGG(\Bk)$ implies $0\leq \BGG(\Bn)\leq \BL_0^{-1}$. So with the choice $\BM=\lang\BGG(\Bn)\rang_{\Bn}$
we have that $\BM\leq \BL_0^{-1}$ which as $\BM\geq 0$ then implies the condition that $\BM\BL_0\BM\leq\BM$. Other choices of $\BM$ may be useful too,
as given $\CK$, both $\BM$ and $\BL_0$ determine the manifold $\CM=W_\BM^{-1}(\CK_0)$, where $\CK_0$ consists
of all self-adjoint maps in $\CK$.
	\end{Arem}
	\section{Exact identities satisfied by the Green's function}	
\setcounter{equation}{0}
Consider a point $\Bx^0$ and take $\Bh(\Bx)$ to be proportional to $\delta$, \ayy{which as conveyed in the Remark \ref{smooth-assumption} denotes a smooth approximation } of a Dirac delta function localized at $\Bx=\Bx^0$:
     \begin{equation}
\label{13}
\Bh(\Bx)=\Bh^0\Gd(\Bx-\Bx^0),\quad \mbox{ with }\Bh^0=-(\BL(\Bx^0)-\BL_0)\Bs^0,
\end{equation}
where the amplitude $\Bs^0\in\CT$ is prescribed. \ayy{We also recall that $\BG$ denotes in this paper an approximate Green's function (see Remark \ref{smooth-assumption}) and here and in the next two sections, we assume that $\BL(\Bx)$ is smooth enough}.

Then, with appropriate decay conditions at infinity
imposed so that the Green's function (fundamental solution) exists and is unique, \eqref{12.1} and \eqref{polarization} informally imply 
     \begin{equation}
\label{14}
\BP(\Bx)=(\BL(\Bx^0)-\BL_0)\Bs^0\Gd(\Bx-\Bx^0)-(\BL(\Bx)-\BL_0)\BG(\Bx,\Bx^0)(\BL(\Bx^0)-\BL_0)\Bs^0.
\end{equation}

With the tensor $\Bs^0\in\CT$ replaced by a succession of $q$ tensors $\Bs^0_1$, $\Bs^0_2$, \ldots, $\Bs^0_q$, each in $\CT$, and 
defining the linear map $\BS^0: {\cal T} \rightarrow  {\cal T}$ via
$\BS^0\Be_i=\Bs^0_i$, we obtain
     \begin{eqnarray}
\label{15}
\mathbb{P}(\Bx)\Be_i & = &(\BL(\Bx^0)-\BL_0)\Bs^0_i\Gd(\Bx-\Bx^0)-(\BL(\Bx)-\BL_0)\BG(\Bx,\Bx^0)(\BL(\Bx^0)-\BL_0)\Bs^0_i \nonumber\\
& = & [(\BL(\Bx^0)-\BL_0)\Gd(\Bx-\Bx^0)-(\BL(\Bx)-\BL_0)\BG(\Bx,\Bx^0)(\BL(\Bx^0)-\BL_0)]\BS^0\Be_i,
\end{eqnarray}
informally implying
     \begin{equation}
\label{16}
\mathbb{P}(\Bx)=\BT(\Bx,\Bx^0)\BS^0,
\end{equation}
with 
\begin{equation}
\label{17}
\BT(\Bx,\Bx^0)=(\BL(\Bx^0)-\BL_0)\Gd(\Bx-\Bx^0)-(\BL(\Bx)-\BL_0)\BG(\Bx,\Bx^0)(\BL(\Bx^0)-\BL_0).
\end{equation}


For fixed $\Bx$ and $\Bx^0$, with $\Bx\ne\Bx^0$ we can consider $\BT(\Bx,\Bx^0)$
as a map from $\CT$ to $\CT$, and given any $\BS^0:\CT\to\CT$ we can choose sources $\Bs^0_i$ such that $\BS^0\Be_i=\Bs^0_i$. Define
$\BT(\Bx,\Bx^0)\CS=\{\BT(\Bx,\Bx^0)\BS~|~\BS\in\CS\}$ where $\BT(\Bx,\Bx^0)\BS$ is the composition of the two maps,
and $\CS$, as defined at the beginning of Sect. 4, is a subspace of $L(\cal T)$ such that $\BA\CS\subset\CK$ for all $\BA\in\CK$.
Then our theorem says that $\BT(\Bx,\Bx^0)\CS\subset\CK$ for all $\Bx$, $\Bx^0$, with $\Bx\ne\Bx^0$. Alternatively we can view 
$\BT(\Bx,\Bx^0)$ as a map from $L(\cal T)$ to $L(\cal T)$ defined as $\BB=\BT(\Bx,\Bx^0)\BA$ iff 
$\Bb_i=\BB\Be_i$ and $\Ba_i=\BA\Be_i$ satisfy $\Bb_i=\BT(\Bx,\Bx^0)\Ba_i$ for $i=1,2,\ldots,q$. Viewed in this way,
$\BT(\Bx,\Bx^0)$ maps $\CS$ to a subset of $\CK$. More generally, \eq{e.4a} implies $\BT(\Bx,\Bx^0)$ maps
$\CS\BD$ to a subset of $\CK\BD$ for all nonsingular $\BD:\CT\to\CT$. Also given any nonsingular $\BS_0$ we can
choose $\BD$ so that $\BS_0\BD\in\CS$ and $\BT(\Bx,\Bx^0)$ will then map this to an element of $\CK\BD$.

To better understand this property of $\BT(\Bx,\Bx^0)$ consider the operator $\BT:\CH\to\CH$, associated with the 
integral kernel $\BT(\Bx,\Bx^0)$, given by
\beq \BT=(\BL-\BL_0)-(\BL-\BL_0)\BG(\BL-\BL_0)=(\BI-\BK\BGY)^{-1}\BK,
\eeq{17.a}
where the last identity follows from \eqref{4}. Define the associated sequence of operators
\beq \BT_\Gl=\Gl(\BI-\lambda \BK\BGY)^{-1}\BK=\sum_{j=0}^{\infty}\lambda^{j+1} (\BK\BGY)^{j}\BK,
\eeq{17.8}
where \ajj{as $\BK$ and $\BGY$ are bounded operators the operator expansion converges
for small enough $\Gl$}.
The associated integral kernel $\BT_\Gl(\Bx,\Bx^0)$ (regarded as a generalized function) can then be written as a series of convolutions, the first 
few terms of which are \ay{given informally by}
\beqa \BT_\Gl(\Bx,\Bx^0) & = & \Gl\Gd(\Bx-\Bx^0)\BK(\Bx^0)+\Gl^2\BK(\Bx)\widehat{\BGY}(\Bx-\Bx^0)\BK(\Bx^0) \nonum
&~& +\Gl^3\int_{\RR^d}\BK(\Bx)\widehat{\BGY}(\Bx-\By_1)\BK(\By_1)\widehat{\BGY}(\By_1-\Bx^0)\BK(\Bx^0)~d\By_1 \nonum
&~& + \Gl^4\int_{\RR^d}\int_{\RR^d}\BK(\Bx)\widehat{\BGY}(\Bx-\By_1)\BK(\By_1)\widehat{\BGY}(\By_1-\By_2)\BK(\By_2)\widehat{\BGY}(\By_2-\Bx^0)\BK(\Bx^0)~d\By_1~d\By_2+\ldots, \nonum
&~&
\eeqa{17.9}
in which $\widehat{\BGY}(\Bx)$, is the Fourier transform of the operator $\BGY(\Bk)$ associated with the operator $\BGY$. Clearly
$\widehat{\BGY}(\Bx)$ lies in the space spanned by the $\BGY(\Bk)$, and therefore \eq{K} implies that each successive term in the expansion
\eq{17.9} lies in $\CK$, and hence $\BT_\Gl(\Bx,\Bx^0)\in\CK$ for small enough $\Gl$, \ajj{assuming the convergence of the series is pointwise not just in the $L^2$ -sense
implied by the boundedness of $\BK$ and $\BGY$.}
\aj{As the properties of  $\BT(\Bx,\Bx^0)$ must be such as to
account for the central theorem 4.1, there presumably must be analytic continuation or other arguments 
which allow us to deduce} that $\BT(\Bx,\Bx^0)\in\CK$ for each $\Bx$ and $\Bx^0$. \aj{We make the assumption that such arguments will be found.} 
\ay{By definition, ${\cal S}$ is a subspace of $L(\cal T)$ 
such that $\BA\cal S \subset \cal K$ for all $\BA\in\cal K$. So as $\BT(\Bx,\Bx^0)$ takes values in $\CK$ we immediately see
that $\BT(\Bx,\Bx^0)$ maps $\CS$ to a subset of $\CK$, as expected. It appears that the converse need not be true as an operator in
$L(\cal T)$ that maps $\CS$ to a subset of $\CK$, generally need not lie in $\CK$. Thus the assertion that $\BT(\Bx,\Bx^0)\in\CK$ for each $\Bx$ and $\Bx^0$
appears to contain more information than that covered by the central theorem 4.1. However, observe that if we consider ${\cal S}$ to be the largest set such that $\BA{\cal S}\subset {\cal K}$ for all $\BA\in{\cal K}$ then ${\cal S}$ becomes a unit algebra. Then, since our theorem implies $\BT(\Bx,\Bx^0)\BS\in{\cal K}$ for all $\BS\in{\cal S}$, by choosing $\BS=\BI$ we obtain $\BT(\Bx,\Bx^0)\in{\cal K}$.}

If $\BL$ is a self adjoint operator then so too is $\BG$. Indeed, supposing that $\BE'=\BG\Bh'$ and $\BE=\BG\Bh$ then \eqref{can-form} implies 
\begin{equation}
\label{18}
	\BJ=\BL\BE-\Bh,\quad \BJ'=\BL\BE'-\Bh' \mbox{ with } \BJ, \BJ'\in {\cal J},\quad \BE, \BE'\in {\cal E},
\end{equation}
and using the orthogonality of $\CE$ and $\CJ$ we have 
\begin{equation}
\label{19}
\langle\BG\Bh,\Bh'\rangle_\CH=\langle \BE,\BL\BE'-\BJ'\rangle_\CH=\langle \BE,\BL\BE'\rangle_\CH
=\langle \BL\BE,\BE'\rangle_\CH=\langle \BL\BE-\BJ,\BE'\rangle_\CH=\langle \Bh,\BG\Bh'\rangle_\CH.
\end{equation}
which implies $\BG$ is self adjoint. In terms of $\BT(\Bx,\Bx^0)$ this says that
\begin{equation}
\label{20}
\BT(\Bx,\Bx^0)=\BT^{\dagger}(\Bx^0,\Bx),
\end{equation}
where $\BT^{\dagger}(\Bx^0,\Bx)$ is the adjoint of $\BT(\Bx^0,\Bx)$ on the space $\CT$. The extension of
$\BT^{\dagger}(\Bx^0,\Bx)$ to an operator (going by the same name)  acting on $L(\cal T)$ is
also the adjoint, with respect to the inner product \eqref{20a},
of the extension of  $\BT(\Bx,\Bx^0)$ that acts on the space $L(\cal T)$. To see this, we have
\beqa \langle \BA, \BT^{\dagger}(\Bx^0,\Bx)\BB \rangle_{L(\cal T)}& = & \sum_{i=1}^q \langle\BA\Be_i,[\BT^{\dagger}(\Bx^0,\Bx)\BB]\Be_i)\rangle_{\CT}
=\sum_{i=1}^q \langle\BA\Be_i,\BT^{\dagger}(\Bx^0,\Bx)(\BB\Be_i)\rangle_{\CT} \nonum
& = & \sum_{i=1}^q \langle\BT(\Bx,\Bx^0)\BA\Be_i,\BB\Be_i\rangle_{\CT}=\langle \BT(\Bx,\Bx^0)\BA, \BB \rangle_{L(\cal T)}.
\eeqa{20b}

Our theorem then implies $\BT^{\dagger}(\Bx^0,\Bx)\CS\BD\subset\CK\BD$ for all $\Bx$, $\Bx^0$, with $\Bx\ne\Bx^0$, and for all nonsingular
$\BD:\CT\to\CT$. By swapping $\Bx$ and $\Bx^0$ we see
that $\BT(\Bx,\Bx^0)\CS\BD$ and $\BT^{\dagger}(\Bx,\Bx^0)\CS\BD$ are both subsets of $\CK$. The latter implies that for all
$\BB\in\CS\BD$ and $\BA\in(\CK\BD)^\perp$ that 
\begin{equation}
\label{21}
0=\langle \BA,\BT^{\dagger}(\Bx,\Bx^0)\BB\rangle_{L(\cal T)}=\langle\BT(\Bx,\Bx^0)\BA,\BB\rangle_{L(\cal T)}.
\end{equation}

So we see that $\BT(\Bx,\Bx^0)$ maps not only $\CS\BD$ to a subset of $\CK\BD$, but also $(\CK\BD)^\perp$ to a subset of $(\CS\BD)^\perp$, in which $(\CS\BD)^{\perp}$ is the orthogonal 
complement of $\CS\BD$ in the space $L(\cal T)$ with respect to the inner product \eq{20a}. 

Further insight into the relation between the operator $\BT$ defined by \eq{17.a} and the Green's operator $\BG$ can be gained by applying the
operator $(\BL-\BL_0)\BGG$ to both sides of
\beq \BJ-\BL_0\BE=(\BL-\BL_0)\BE-\Bh,\quad \BJ\in\CJ,\quad\BE\in\CE, \eeq{A.1}
where we now only require that $\Bh\in\CH$. This gives
\beq -(\BL-\BL_0)\BE=(\BL-\BL_0)\BGG(\BL-\BL_0)\BE-(\BL-\BL_0)\BGG\Bh, \eeq{A.2}
and hence 
\beq (\BL-\BL_0)\BE=[\BI+(\BL-\BL_0)\BGG]^{-1}(\BL-\BL_0)\BGG\Bh. \eeq{A.3}
As $\BE=\BG\Bh$ we see that the Green's function $\BG$ satisfies
\beqa (\BL-\BL_0)\BG & = & [\BI+(\BL-\BL_0)\BGG]^{-1}(\BL-\BL_0)\BGG \nonum
                    & = & (\BI-\BK\BGY)^{-1}\BK\BGG \nonum
                    & = & \BT\BGG,
\eeqa{A.4}
where we have made use of \eqref{9} (with $\Gl=1$). Instead of \eq{extgov} we may write
\beq \mathbb{J}(\Bx)=\BL(\Bx)\mathbb{E}(\Bx)-\mathbb{H}(\Bx),\mbox{ with }\mathbb{E}\in\mathfrak{E},\quad \mathbb{J}\in\mathfrak{J},
\eeq{A.5}
and we see that if $\BGG\mathbb{H}$ takes values in $\CS\BD$, then $(\BL-\BL_0)\mathbb{E}$ takes values in  $\CK\BD$. The constraint
that $\BGG\mathbb{H}$ takes values in $\CS\BD$ is of course a non-local constraint on $\mathbb{H}(\Bx)$ and therefore not
as easy to check as the constraint that $\mathbb{S}(\Bx)\in\CS\BD$ (assuming 
$\mathbb{H}=-(\BL-\BL_0)\mathbb{S}$ for some $\mathbb{S}\in\mathfrak{H}$).

	\section{Links between Green's functions of different physical problems}
	\setcounter{equation}{0}

In the same way that the theory of exact relations for composites easily provides links between effective tensors, so too does our theory easily provide links 
between the Green's functions of different physical problems. The treatment here adapts the theory of links for composites, given in
\cite{Grabovsky:2000:ERE} Sect. 4.3, to Green's functions.

Consider $m$ different physical problems, each described by an equation that can be expressed
in the form \eq{I.1}:
\beq \BJ^{(i)}(\Bx)=\BL^{(i)}(\Bx)\BE^{(i)}(\Bx)-\Bh^{(i)}(\Bx), \mbox{ with }\BJ^{(i)}\in\CJ^{(i)},\quad\BE^{(i)}\in\CE^{(i)},\quad\Bh^{(i)}\in\CH^{(i)}, \eeq{l.1}
where $i=1,2,\ldots,m$ indexes each different problem, and each field $\BJ^{(i)}(\Bx)$, $\BE^{(i)}(\Bx)$ and $\Bh^{(i)}(\Bx)$ takes values in a tensor
space $\CT^{(i)}$ for every $\Bx\in\RR^d$. The projection $\BGG_1^{(i)}$ onto $\CE^{(i)}$ is assumed to act locally in Fourier space, i.e., if $\BE^{(i)}=\BGG_1^{(i)}\BA^{(i)}$ then the Fourier components $\widehat{\BE}^{(i)}(\Bk)$ and $\widehat{\BA}^{(i)}(\Bk)$
of $\BE^{(i)}$ and $\BA^{(i)}$ are related via $\widehat{\BE}^{(i)}(\Bk)=\BGG_1^{(i)}(\Bk)\widehat{\BA}^{(i)}(\Bk)$ for some operator $\BGG_1^{(i)}(\Bk)$ that projects 
onto a subspace $\CE^{(i)}_{\Bk}\subset\CT$.

We can rewrite this set of equations in the equivalent form
\beq \underbrace{\bpm \BJ^{(1)}(\Bx)\\ \BJ^{(2)}(\Bx)\\ \vdots \\ \BJ^{(m)}(\Bx) \epm}_{\BJ(\Bx)}
=\underbrace{\bpm \BL^{(1)}(\Bx) &  0 & \hdots & 0 \\
      0        & \BL^{(2)}(\Bx) & \hdots & 0 \\
      \vdots   & \vdots   & \ddots & \vdots \\
      0        & 0        & \hdots & \BL^{(m)}(\Bx) \epm}_{\BL(\Bx)}
\underbrace{\bpm \BE^{(1)}(\Bx)\\\BE^{(2)}(\Bx)\\ \vdots \\\BE^{(m)}(\Bx) \epm}_{\BE(\Bx)} -
\underbrace{\bpm \Bh^{(1)}(\Bx)\\ \Bh^{(2)}(\Bx)\\ \vdots \\ \Bh^{(m)}(\Bx) \epm}_{\Bh(\Bx)}, 
\eeq{l.2}
where $\BJ(\Bx)$, $\BE(\Bx)$ and $\Bh(\Bx)$ take values in $\CT=\CT^{(1)}\oplus\CT^{(2)}\oplus\ldots\oplus\CT^{(m)}$, and satisfy
\beq \BE\in\CE, \quad \BJ\in\CJ, \quad \Bh\in\CH,
\eeq{l.2a}
in which $\CJ$ and $\CE$ consist of all those fields $\BJ(\Bx)$ and $\BE(\Bx)$, respectively, taking the form indicated in \eq{l.2}
with component fields $\BJ^{(i)}\in \CJ^{(i)}$ and $\BE^{(i)}\in \CE^{(i)}$, for $i=1,2,\ldots,m$, and $\CH$ is defined by $\CH=\CE\oplus\CJ$. The Green's
function for the system of uncoupled equations \eq{l.2} and\eq{l.2a} takes the form
\beq \BG(\Bx,\Bx^0)=\bpm \BG^{(1)}(\Bx,\Bx^0) &  0 & \hdots & 0 \\
      0        & \BG^{(2)}(\Bx,\Bx^0) & \hdots & 0 \\
      \vdots   & \vdots   & \ddots & \vdots \\
      0        & 0        & \hdots & \BG^{(m)}(\Bx,\Bx^0) \epm,
\eeq{l.3}
in which $\BG^{(i)}(\Bx,\Bx^0)$ denotes the Green's function for the ``$i$-th'' problem. We introduce a constant reference tensor $\BL_0$ 
and a constant tensor $\BM$ that are both assumed to be block diagonal:
\beq \BL_0=\bpm \BL^{(1)}_0 &  0 & \hdots & 0 \\
      0        & \BL^{(2)}_0 & \hdots & 0 \\
      \vdots   & \vdots   & \ddots & \vdots \\
      0        & 0        & \hdots & \BL^{(m)}_0 \epm,\quad \BM=\bpm \BM^{(1)} &  0 & \hdots & 0 \\
      0        & \BM^{(2)} & \hdots & 0 \\
      \vdots   & \vdots   & \ddots & \vdots \\
      0        & 0        & \hdots & \BM^{(m)} \epm.
\eeq{l.4}
Then associated with the decoupled system \eq{l.2} and \eq{l.3} is an operator
\beq \BGY(\Bk)=\underbrace{\bpm \BGY^{(1)}(\Bk) &  0 & \hdots & 0 \\
      0        & \BGY^{(2)}(\Bk) & \hdots & 0 \\
      \vdots   & \vdots   & \ddots & \vdots \\
      0        & 0        & \hdots & \BGY^{(m)}(\Bk) \epm}=\underbrace{\bpm \BM^{(1)}-\BGG^{(1)}(\Bk) &  0 & \hdots & 0 \\
      0        & \BM^{(2)}-\BGG^{(2)}(\Bk) & \hdots & 0 \\
      \vdots   & \vdots   & \ddots & \vdots \\
      0        & 0        & \hdots & \BM^{(m)}-\BGG^{(m)}(\Bk) \epm},
\eeq{l.5}
in which 
\beq \BGG^{(i)}(\Bk)=[\BGG^{(i)}_1(\Bk)\BL^{(i)}_0\BGG^{(i)}_1(\Bk)]^{-1}\BGG^{(i)}_1(\Bk), \eeq{l.5aa}
where $\BGG^{(i)}_1(\Bk)$ is the projection  onto $\CE^{(i)}_\Bk$ and the operator inverse is taken on $\CE^{(i)}_\Bk$.

\ay{As before, we search for subspaces $\CK\subset L(\cal T)$ having the property \eq{K} and associated subspaces $\CS\subset L(\cal T)$ that are 
subspaces having the property that $\BA\cal S \subset \cal K$ for all $\BA\in\cal K$. Due to the special algebraic structure of the problem, the search
for such subspaces simplifies: see Section 4.3 of \cite{Grabovsky:2000:ERE} and pages 5--3 to 5--11 of \cite{Grabovsky:2016:CMM}}. The exact relation then implies that
\beq (\BL(\Bx)-\BL_0)\BG(\Bx,\Bx^0)(\BL(\Bx^0)-\BL_0)\ayy{+\BL_0-\BL}\in\CK \text{ for all } \Bx, \Bx^0, \text{ with }\Bx\ne\Bx^0,
\eeq{l.5a}
and if $\BL^{(i)}(\Bx)$ is self-adjoint for all  $i$ we have additionally that for all nonsingular $\BD\to\BD$,
\beq ((\BL(\Bx)-\BL_0)\BG(\Bx,\Bx^0)(\BL(\Bx^0)-\BL_0)\ayy{+\BL_0-\BL})(\CK\BD)^\perp\subset(\CS\BD)^\perp \text{ for all } \Bx, \Bx^0, \text{ with }\Bx\ne\Bx^0,
\eeq{l.5b}
in which $(\CK\BD)^\perp$ and $(\CS\BD)^\perp$ are the orthogonal complements of $\CK\BD$ and $\CS\BD$ on the space $L(\cal T)$, with respect to the inner product \eqref{20a}.

For these exact relations to a  generate a link between the Green's functions of the different problems it is necessary that $\CK$ not be separable. It is separable
if after some reordering of the indices $i=1,2,\ldots,m$ labeling the $m$ problems there exists a subdivision of the problems such that the exact
relation decouples.  In other words, there exists a $p$, $1<p<m$, such that there are subspaces $\CK_1\subset L({\cal T}_1)$, and
$\CK_2\subset L({\cal T}_2)$, where
\beq \CT_1=\CT^{(1)}\oplus\CT^{(2)}\oplus\ldots\oplus\CT^{(p)},\quad \CT_2=\CT^{(p+1)}\oplus\CT^{(2)}\oplus\ldots\oplus\CT^{(m)}, 
\eeq{l.6}
with $\CK_1$ and $\CK_2$ each having the algebraic properties of an exact relation, and any block diagonal tensor
\beq \BB=\bpm \BB^{(1)} &  0 & \hdots & 0 \\
      0        & \BB^{(2)} & \hdots & 0 \\
      \vdots   & \vdots   & \ddots & \vdots \\
      0        & 0        & \hdots & \BB^{(m)} \epm
\eeq{l.7}
is in $\CK$ if and only if 
\beq \BB_1=\bpm \BB^{(1)} &  0 & \hdots & 0 \\
      0        & \BB^{(2)} & \hdots & 0 \\
      \vdots   & \vdots   & \ddots & \vdots \\
      0        & 0        & \hdots & \BB^{(p)} \epm\in\CK_1~~\text{and}~~
     \BB_2=\bpm \BB^{(p+1)} &  0 & \hdots & 0 \\
      0        & \BB^{(p+2)} & \hdots & 0 \\
      \vdots   & \vdots   & \ddots & \vdots \\
      0        & 0        & \hdots & \BB^{(m)}  \epm\in\CK_2.
\eeq{l.8}
When we say $\CK_1$ and $\CK_2$ have the properties of an exact relation we specifically mean that
for all $\Bk\in\RR^d$ one has
\beq \BA_1\BGY_j(\Bk)\BA_2\in\CK_j,~~{\rm for~all~}\BA_1,\BA_2\in\CK_j,~~j=1,2,
\eeq{l.9}
in which
\beq \BGY_1(\Bk)=\bpm \BGY^{(1)}(\Bk) &  0 & \hdots & 0 \\
      0        & \BGY^{(2)}(\Bk) & \hdots & 0 \\
      \vdots   & \vdots   & \ddots & \vdots \\
      0        & 0        & \hdots & \BGY^{(p)}(\Bk) \epm ~~\text{and}~~
     \BGY_2(\Bk)=\bpm \BGY^{(p+1)}(\Bk) &  0 & \hdots & 0 \\
      0        & \BGY^{(p+2)}(\Bk) & \hdots & 0 \\
      \vdots   & \vdots   & \ddots & \vdots \\
      0        & 0        & \hdots & \BGY^{(m)}(\Bk)  \epm.
\eeq{l.10}

\section{Exact identities satisfied by the DtN map and Boundary Field Equalities: A generalization of conservation laws}	
\setcounter{equation}{0}
\aj{This section generalizes the ideas developed in \cite{Thaler:2014:EDV}, where it was shown how Hill's exact relation in the
theory of composites, could be used to derive exact identities satisfied by the ``Dirichlet-to-Neumann map'' of 
a body $\GO$ containing two elastically isotropic materials with the same shear modulus:
in particular, these identities allow one to exactly deduce the volume fractions occupied 
by the phases from boundary measurements. The key idea was to apply (non-local) boundary conditions on the boundary tractions and displacements on the boundary $\Md\GO$ of $\GO$
in such a way that they mimic the body placed in an appropriate infinite medium with appropriate sources outside.
The exact relations satisfied by the fields in the latter problem imply that the fields inside $\GO$ satisfy these exact relations too, and this
in turn allows one to obtain additional information about the boundary fields: the boundary field equalities.}

\ayy{We recall the conventions of Remark \ref{smooth-assumption}, and our assumption made in the beginning of  Section 5 that $\BL(\Bx)$ is smooth enough. We also mention that here, and in the next section we assume $\Omega$ is sufficiently smooth} and we restrict attention to those equations (having the required canonical form)
for which the response of a body filled by inhomogeneous material and devoid of sources inside
is governed by a ``Dirichlet-to-Neumann map'' (DtN map) $\BGL_\GO$. Symbolically we may
write 
\beq \Md\BJ=\BGL_\GO(\Md\BE), \eeq{b.1}
where $\Md\BE$ \aj{informally} denotes the boundary information associated with the field $\BE$, and
$\Md\BJ$ \aj{informally} denotes the boundary information associated with the field $\BJ$. In the context of the \aj{primary} equations
\eq{b.0}, $\Md\BE$ represents the value of the potential field $\Bu(\Bx)$ at the boundary $\Md\GO$ of $\GO$ and
$\Md\BJ$ represents the value of the flux vector field $\Bn\cdot\BJ(\Bx)$ at the boundary $\Md\GO$, where $\Bn$ is the outwards normal to the surface $\Md\GO$ (assumed smooth).
(Equivalently, for the \aj{primary} equations, $\Md\BE$ can be taken as the tangential values of $\BE(\Bx)$ since integrating these over the surface yields the boundary values of  $\Bu(\Bx)$, up to an
additive constant vector). \aj{The appendix gives further examples of boundary fields $\Md\BJ$ and $\Md\BE$ that are associated with various physical equations, in particular
quasistatic equations. Beyond \eq{b.0} and the equations in the appendix we avoid giving a precise definition of the fields $\Md\BE$ and $\Md\BJ$.}

Specifying $\Md\BE$ determines a unique $\BE$ (and hence $\BJ$) that solve
\beq \BJ(\Bx)=\BL(\Bx)\BE(\Bx),\quad \BJ\in\overline{\CJ}_\GO,\quad \BE\in\overline{\CE}_\GO.
\eeq{b.2}
where \aj{$\overline{\CJ}_\GO$ and $\overline{\CE}_\GO$ are the closures of} the spaces $\CJ_\GO$ and $\CE_\GO$
comprised of those fields $\BJ(\Bx)$ and $\BE(\Bx)$ defined in $\GO$ which can be extended outside $\GO$ in such a way  
that with their extensions they lie in $\CJ$ and $\CE$, respectively. \aj{Note that  $\overline{\CJ}_\GO$ and $\overline{\CE}_\GO$ are not orthogonal, nor even nonintersecting.
Associated with $\BJ(\Bx)$ is the boundary field $\Md\BJ$. As it depends linearly on $\Md\BE$, this linear relation defines the DtN map \eq{b.1}.}

\aj{We make a side remark that, as shown in \cite{Milton:2016:ETC}, Chap.3 \eq{b.2} can be reformulated as an equation in the abstract theory of composites, 
\beq \underbrace{\BL_0^{-1/2}\BJ(\Bx)}_{\widetilde{\BJ}}
=\underbrace{[\BL_0^{-1/2}\BL(\Bx)\BL_0^{-1/2}]}_{\widetilde{\BL}}\underbrace{\BL_0^{1/2}\BE(\Bx)}_{\widetilde{\BE}},
\quad\widetilde{\BJ}\in\widetilde{\CU}\oplus\widetilde{\CE}=\BL_0^{-1/2}\overline{\CJ}_\GO,\quad \widetilde{\BE}\in\widetilde{\CU}\oplus\widetilde{\CE}=\BL_0^{1/2}\overline{\CE}_\GO,
\eeq{b.2a}
where $\BL_0$ can be any positive definite self-adjoint tensor (the ``reference tensor''), and
\beq \widetilde{\CU}\equiv(\BL_0^{1/2}\overline{\CE}_\GO)\cap(\BL_0^{-1/2}\overline{\CJ}_\GO),\quad
\widetilde{\CJ}=\BL_0^{-1/2}{\CJ}_\GO^0,\quad\widetilde{\CE}=\BL_0^{-1/2}{\CE}_\GO^0,
\eeq{b.2b}
in which ${\CJ}_\GO^0$ and ${\CE}_\GO^0$ consist of those fields in $\CJ$ and $\CE$, respectively, that vanish outside $\GO$. The DtN map is then associated with the ``effective operator'' $\BL_*$ mapping $\widetilde{\CU}$ to itself. Fields in $\widetilde{\CU}$ can be uniquely characterized either by the associated value of $\Md\BE$ or by the associated
value of $\Md\BJ$: given $\Md\BE$ there is a map $\Psi$ to a field $\widetilde{\BU}=\Psi(\Md\BE)$ in $\widetilde{\CU}$, and given $\BL_*\widetilde{\BU}\in\widetilde{\CU}$ 
there is map $\Phi$ to $\Md\BJ=\Phi\BL_*\widetilde{\BU}=\Phi\BL_*\Psi(\Md\BE)$. This provides the connection between $\BL_*$ and the DtN map: $\BGL_\GO=\Phi\BL_*\Psi$.
A similar reformulation of \eq{b.2} as a problem in the abstract theory of composites
was made independently by Grabovsky (see \cite{Grabovsky:2016:CMM}, Chap.2) who takes $\BL_0=\BI$ and refers to fields in $\CU$ as ``harmonic functions'' (as they are indeed harmonic functions in the conductivity
problem). We thank Yury Grabovsky for pointing out the need for replacing  $\CJ_\GO$ and $\CE_\GO$ by their closures $\overline{\CJ}_\GO$ and $\overline{\CE}_\GO$ in \eq{b.2},
unless, of course, they are already closed.}


Boundary field equalities may be viewed as exact identities satisfied by the DtN map that are independent
of the precise microstructure inside the body. For the boundary field equalities derived here we need only assume that $\BL(\Bx)\in\CM$
and satisfies the coercivity condition \eq{boundcoer}
for all $\Bx\in\GO$. 

Specifying $\Md\BE$ is one of many possible boundary conditions that uniquely determine the fields inside $\GO$. Another
frequently used one is specifying $\Md\BJ$. A different sort, that we use here, is providing some type of mixed non-local boundary condition
that mimics surrounding the body by infinite homogeneous medium with tensor $\BL_1$ with appropriate sources placed
outside the body. Sources outside the body can be considered as a superposition of localized delta-function sources. So let us
define $\GO^C\equiv \RR^d\setminus\GO$ and consider a single delta function source at $\Bx_0\in\GO^C$.

Then in $\GO^C$ the equations \ay{informally} take the form
\beq \BJ(\Bx)=\BL_1\BE(\Bx)+(\BL_1-\BL_0)\Bs^0\Gd(\Bx-\Bx_0),\quad \BJ\in\CJ_{\GO^C},\quad \BE\in\CE_{\GO^C},
\eeq{b.3}
where $\CJ_{\GO^C}$ and $\CE_{\GO^C}$ are comprised of those fields $\BJ(\Bx)$ and $\BE(\Bx)$ defined outside $\GO$ which can be extended inside $\GO$ in such a way  
that with their extensions they lie in $\CJ$ and $\CE$, respectively. Now one can easily (numerically if not analytically) solve for the problem of a point
source in \ay{an infinite} homogeneous medium \ay{having moduli $\BL_1$}:
\beq \BJ^0(\Bx)=\BL_1\BE^0(\Bx)+(\BL_1-\BL_0)\Bs^0\Gd(\Bx-\Bx_0),\quad \BJ^0\in\CJ,\quad \BE^0\in\CE.
\eeq{b.4}
\ajj{Alternatively, the analysis of this section will go through if we replace the delta function with any square-integrable source that is compactly supported outside $\GO$,
say with $\Bs^0\Gd(\Bx-\Bx_0)$ replaced by $\Bs^0r(\Bx)$ where the scalar valued function $r(\Bx)$ is square integrable and $r(\Bx)=0$ inside $\GO$.}

Subtracting \eq{b.4} from \eq{b.3} gives \ay{informally} 
 \beq \widetilde{\BJ}(\Bx)=\BL_1\widetilde{\BE}(\Bx),\quad \widetilde{\BJ}\in\CJ_{\GO^C},\quad \widetilde{\BE}\in\CE_{\GO^C},\quad\widetilde{\BJ}(\Bx)=\BJ(\Bx)-\BJ^0(\Bx),\quad
\widetilde{\BE}(\Bx)=\BE(\Bx)-\BE^0(\Bx).
\eeq{b.5}
The boundary information $\Md\widetilde{\BJ}$ and $\Md\widetilde{\BE}$ associated with the fields $\widetilde{\BJ}$ and $\widetilde{\BE}$ are
linked by the exterior DtN map $\BGL_{\GO^C}$:  $\BGL_{\GO^C}(\Md\widetilde{\BE})=\Md\widetilde{\BJ}$. As the material outside $\GO$ is 
homogeneous, $\BGL_{\GO^C}$ is in principle computable and will be assumed known. Since the fields outside and inside the body
must be compatible, i.e., share the same boundary information $\Md\BJ$ and $\Md\BE$, we see that $\Md\BE$ and $\Md\BJ$ must be such that \ay{informally}
\beq \Md\BJ-\Md\BJ^0=\BGL_{\GO^C}(\Md\BE-\Md\BE^0),\quad \text{i.e., } \Md\BJ-\BGL_{\GO^C}(\Md\BE)=\Md\BJ^0-\BGL_{\GO^C}(\Md\BE^0),
\eeq{b.6}
in which $\Md\BJ^0$ and $\Md\BE^0$ denote the boundary information associated with $\BJ^0$ and $\BE^0$. Some care needs to be taken. For example
in the conductivity problem in which $\BJ(\Bx)$ is the current, $\Md\BJ$ represents the flux $\BJ\cdot\Bn$ where $\Bn$ is the outward normal
to $\GO$. So in defining $\BGL_{\GO^C}$ it is important that it maps to the flux $\BJ\cdot\Bn$ where again $\Bn$ is the outward normal
to $\GO$, not the outward normal to $\GO^C$. 

Equation \eq{b.6} provides the needed boundary conditions that constrain $\Md\BJ$ and $\Md\BE$. They are not so pleasant as they are
non-local and involve $\BGL_{\GO^C}$. Instead of specifying $\Md\BE$ or $\Md\BJ$ one specifies $\Md\BJ-\BGL_{\GO^C}(\Md\BE)$.
\aj{Supposing that the DtN map $\BGL_\GO$ entering \eq{b.1} has been experimentally measured, 
or numerically calculated, then \eq{b.6} provides the explicit equation
\beq \BGL_{\GO}(\Md\BE)-\BGL_{\GO^C}(\Md\BE)=\Md\BJ^0-\BGL_{\GO^C}(\Md\BE^0) \eeq{b.6a}
that can be solved for $\Md\BE$. The question arises as to whether a solution exists, and if so, is it unique?}
However, it is exactly the same as solving the governing equations \eq{I.1} for
the body $\GO$ surrounded by homogeneous medium with tensor $\BL_1$ and with the source term $\Bh(\Bx)=(\BL_1-\BL_0)\Bs^0\Gd(\Bx-\Bx_0)$, so uniqueness
of $\BE$ (and hence $\Md\BE$ and $\Md\BJ$) is assured.
 
Now instead of considering a single experiment, one can consider $q$ experiments with $\Bs^0$ replaced by $\Bs_1^0,\Bs_2^0,\ldots,\Bs_q^0$, each source
remaining  at $\Bx_0$. The fields $\BE(\Bx)$ and $\BJ(\Bx)$ are then replaced by $\BE_i(\Bx)$ and $\BJ_i(\Bx)$. Let us define, as previously,
\beq \BS^0\Be_i=\Bs^0_i,\quad \mathbb{E}(\Bx)\Be_i=\BE_i(\Bx),\quad \mathbb{J}(\Bx)\Be_i=\BJ_i(\Bx), \quad \mathbb{P}(\Bx)=\mathbb{J}(\Bx)-\BL_0\mathbb{E}(\Bx),
\eeq{b.7}
and let $\Md\mathbb{E}$ and $\Md\mathbb{J}$ represent the boundary information associated with $\mathbb{E}(\Bx)$ and $\mathbb{J}(\Bx)$ respectively.
If $\Bs_1^0,\Bs_2^0,\ldots,\Bs_q^0$ are chosen so that $\BS^0\in\CS\BD$ for some nonsingular $\BD: \CT\to\CT$ then our theorem implies that
$\mathbb{P}(\Bx)\in\CK\BD$ for all $\Bx\in\GO$. If $\CS$ contains a non-singular element $\BS^1$ then there is no restriction on $\BS^0$ as we can
choose $\BD=(\BS^1)^{-1}\BS^0$. This then constrains $\mathbb{J}(\Bx)-\BL_0\mathbb{E}(\Bx)$ to lie in $\CK\BD$
for all $\Bx\in\GO$. This should then naturally provide constraints on the boundary information $\Md\mathbb{E}$ and $\Md\mathbb{J}$,
thus yielding exact identities satisfied by the DtN map. \aj{The examples in the next section demonstrate this explicitly.}
These exact identities must be satisfied for every choice of nonsingular $\BD: \CT\to\CT$, every choice 
of $\Bx_0$ outside $\GO$, and
for every choice of $\BS^0\in\CS\BD$. \aj{It seems unlikely that the set of exact identities produced as $\Bx_0$ varies outside $\GO$ will
all be independent. Rather it is probably the case that a source at $\Bx_0$ outside $\GO$ has the same effect as a set of sources
around the boundary of $\GO$, implying that it suffices to take sources restricted to the boundary of $\GO$. (Technically, one should
take sources just outside $\GO$ and let them approach the boundary).}

If $\BL(\Bx)$ is self-adjoint one has the additional constraint that if
$\BS^0\in(\CK\BD)^\perp$ then $\mathbb{J}(\Bx)-\BL_0\mathbb{E}(\Bx)$ must lie in $(\CS\BD)^\perp$ for all $\Bx\in\GO$, and
this then implies additional constraints on the boundary information $\Md\mathbb{E}$ and $\Md\mathbb{J}$.

Depending on the nature of the exact relations, boundary field equalities may involve the volume fraction of one phase in a body containing two phases. 
and thus allow the volume fraction to be exactly determined (e.g., see \cite{Thaler:2014:EDV}).

Note that the constitutive law inside the body $\GO$ that $\mathbb{J}(\Bx)=\BL(\Bx)\mathbb{E}(\Bx)$ for some $\BL(\Bx)$ such that $W_{\BM}(\BL(\Bx))\in\CK\BD$
for all $\Bx$ [or more generally that $\mathbb{J}(\Bx)=\BL(\Bx)\mathbb{E}(\Bx)+(\BL(\Bx)-\BL_0)\mathbb{S}(\Bx)$ inside $\GO$ for some
$\BL(\Bx)$ with $W_{\BM}(\BL(\Bx))\in\CK$ and for some $\mathbb{S}(\Bx)\in\CS$] can just be viewed as a non-linear constraint on the fields $\mathbb{J}$
and $\mathbb{E}$ inside $\GO$. This constraint, coupled with the appropriate boundary conditions, allows one to deduce the boundary field
equalities that constrain the boundary information $\Md\mathbb{E}$ and $\Md\mathbb{J}$. From this perspective our analysis immediately
applies to a large range of non-linear problems. \aj{This is developed further in the next section}.
 

\section{A alternative view of some Boundary Field Equalities}	
\setcounter{equation}{0}

As an interesting illustration of boundary field inequalities, expressed in a different form, let us focus on the equations \eq{b.0}
with initially $h_{i\Ga}(\Bx)=0$ within $\GO$. We will concentrate on cases $d=2$ and $d=3$ since these are of greatest practical importance.
First, starting in the case $d=2$, and defining new fields, $\widetilde{J}_{j\Gb}(\Bx)$, 
with components given by
\beq \widetilde{J}_{1\Gb}(\Bx)=-E_{2\Gb}(\Bx),\quad \widetilde{J}_{2\Gb}(\Bx)=E_{1\Gb}(\Bx), \eeq{y.1}
the equations \eq{b.0} can be rewritten as 
\beq \BJ(\Bx)=\BL(\Bx)\CR^\perp\widetilde{\BJ}(\Bx),\quad \Div\BJ=0,\quad \Div\widetilde{\BJ}=0, \eeq{y.2}
with $\CR^\perp$ being a 4-index tensor with elements
\beq  \{\CR^\perp\}_{i\Gb j\Gg}=\Gd_{\Gb\Gg}\{\BR^\perp\}_{ij},\quad \BR^\perp=\bpm 0 & 1 \cr -1 & 0 \epm,
\eeq{y.3}
where $\BR^\perp$ is the matrix for a  $90^\circ$ rotation. Thus,
the action of $\CR^\perp$ on $\widetilde{\BJ}(\Bx)$ is to rotate its columns by $90^\circ$, each column being a divergence-free 
field, to give $\BE(\Bx)$, each column then being a curl-free field. We can also consider $2m$ solutions to \eq{y.2}
labeled by an index pair $k\nu$, $k=1,2$, $\nu=1,2,\ldots,m$, with fields $\BJ^{k\nu}(\Bx)$ and $\widetilde{\BJ}^{k\nu}(\Bx)$
each satisfying \eq{y.2}, so that the resulting set of equations can then be rewritten as
\beq \mathbb{J}(\Bx)=\BL(\Bx)\CR^\perp\widetilde{\mathbb{J}}(\Bx), \quad \Div\mathbb{J}=0, \quad \Div\widetilde{\mathbb{J}}(\Bx)=0,
\eeq{y.4}
where $\mathbb{J}(\Bx)$ and $\widetilde{\mathbb{J}}(\Bx)$ are four-index fields with elements
\beq \mathbb{J}_{i\Ga k\nu}(\Bx)={J}_{i\Ga}^{k\nu}(\Bx),\quad \widetilde{\mathbb{J}}_{i\Ga k\nu}(\Bx)=\widetilde{J}_{i\Ga}^{k\nu}(\Bx),
\eeq{y.5}
and the divergence acts on the first index of these fields. The constitutive law in \eq{y.4}, the boundedness and coercivity
of $\BL$, and the fact that $\BL(\Bx)$ is constrained to take values in $\CM$, 
can be replaced by the constraint that $\BQ(\Bx)$, defined to have elements
\beq Q_{i\Ga k\nu 1}(\Bx)=\mathbb{J}_{i\Ga k\nu}(\Bx),\quad Q_{i\Ga k\nu 2}(\Bx)=\widetilde{\mathbb{J}}_{i\Ga k\nu}(\Bx),
\eeq{y.6}
takes values in a set $\CB$, where a 5-index tensor $\BQ$ is defined to lie in $\CB$ if and only if there exists a four-index tensor $\BL\in\CM$
satisfying the boundedness and coercivity constraint that $\Gb_0\BI\geq\BL\geq\Ga_0\BI>0$, and
\beq Q_{i\Ga k\nu 1}=L_{i\Ga j\Gb}R^\perp_{j\Gb \ell\Gg}Q_{\ell\Gg k\nu 2}, \eeq{y.7}
where (following the Einstein summation convention) sums over repeated indices are assumed. If the manifold $\CM$ has dimension $m_0$
then $\CB$ will be a $r$-dimensional object (in a $8m^2$-dimensional space) with $r=4m^2+m_0$. 

At points $\Bx$ where $\mathbb{J}(\Bx)$ is non-singular
this constraint that $\BQ(\Bx)\in\CB$ is equivalent to requiring that the fields satisfy the clearly non-linear constraint,
\beq \BL(\Bx)\equiv -\mathbb{J}(\Bx)[\widetilde{\mathbb{J}}(\Bx)]^{-1}\CR^\perp\in\CM, \eeq{y.8}
and that $\BL(\Bx)$ thus defined satisfies the boundedness and coercivity condition \eq{boundcoer}.
We can of course lump all the last four indices of $\BQ(\Bx)$ into a single index $\Gg$ taking values from
$1$ to $s=4m^2$, and then regard $\BQ(\Bx)$ as a $2\times s$ matrix valued field satisfying $\Div\BQ=0$,
i.e. the columns of $\BQ$ are then divergence-free vector fields. The mapping of the last four indices onto $\Gg$ can be chosen
so the first $2m^2$ columns of $\BQ$ (are associated with $\mathbb{J}$ while the last  $2m^2$ columns of $\BQ$ are associated with
$\widetilde{\mathbb{J}}$. We let $\BQ_1$ and $\BQ_2$ denote the $2\times 2m^2$-matrix valued fields formed by the first and last
set of $2m^2$ columns, so these are associated with $\mathbb{J}$ and $\widetilde{\mathbb{J}}$, respectively.
Associated with $\BQ(\Bx)$, $\BQ_1(\Bx)$, and $\BQ_2(\Bx)$ are then
also the fluxes $\Bq(\Bx)=\Bn(\Bx)\cdot\BQ(\Bx)$, $\Bq_1(\Bx)=\Bn(\Bx)\cdot\BQ_1(\Bx)$, $\Bq_2(\Bx)=\Bn(\Bx)\cdot\BQ_2(\Bx)$
at the boundary $\Md\GO$, where $\Bn$ is the outwards normal. Thus the first and last $2m^2$ elements of $\Bq(\Bx)$ give
$\Bq_1(\Bx)$ and $\Bq_2(\Bx)$ respectively.

The boundary information for the equations \eq{b.0} is usually taken as the values of the potential $\Bu$ and
flux $\Bn\cdot\BJ$, where $\Bn$ is the surface normal. However, as noted in the previous section, specifying the tangential value of $\Grad\Bu$ 
at the surface allows one to recover $\Bu$ (up to a trivial constant), and thus is an equivalent condition.
So we see that instead of using the potential $\Bu$ as our boundary information one can use $\Bn\cdot\widetilde{\BJ}$.
Then \eq{b.6} can be interpreted as a non-local, linear constraint on the boundary flux $\Bq$. It can be rewritten as
\beq \Bq_1-\BGL_{\GO^C}(\Bq_2)=\Bq_1^0-\BGL_{\GO^C}(\Bq_2^0), \eeq{y.9}
where $\Bq_1^0$ and $\Bq_2^0$ are the fluxes across $\Md\GO$ associated with solving the equations in an infinite homogeneous medium with constant tensor $\BL_1\in\CM$
with a source $\mathbb{S}(\Bx)$ with support outside $\GO$ taking values in $\CS\BD$, and with $\BGL_{\GO^C}$ 
being the exterior DtN map associated with this medium. 
Our theorem then forces $\mathbb{J}(\Bx)-\BL_0\mathbb{E}(\Bx)$ within $\GO$ to lie in $\CK\BD$ and this constraint can alternatively be 
written as $\BQ(\Bx)\in\CD$  implying $\Bq(\Bx)\in\Bn(\Bx)\CD$ . To precisely define $\CD$  one can return to the representation 
where $\BQ$ is a 5-index tensor. Then $\BQ\in\CD$  if and only if $\BZ\in\CK\BD$ where $\BZ$ is the four index tensor with elements
\beq Z_{i\Ga s\nu}=Q_{i\Ga s\nu 1}(\Bx)-\{\BL_0\}_{i\Ga j\Gb}R^\perp_{j\Gb k\Gg}Q_{k\Gg s\nu 2}.
\eeq{y.10}
The constraint that $\Bq(\Bx)\in\Bn(\Bx)\CD$  for all $\Bx\in\Md\GO$ is a boundary field equality.

We can relax the constraint that $\mathbb{S}(\Bx)$ is zero inside $\GO$, and given $\BD$, $\BL_0$, $\Ga_0$, $\Gb_0$
redefine $\CB$ so that a 5-index tensor $\BQ$ is in $\CB$ if and only if there exists an $\BS\in\CS\BD$ and a
$\BL\in\CM$ with $\Gb_0\BI\geq\BL\geq\Ga_0\BI>0$, such that
\beq Q_{i\Ga s\nu 1}+\{\BL_0\}_{i\Ga j\Gb}S_{j\Gb s\nu}=L_{i\Ga j\Gb}[R^\perp_{j\Gb k\Gg}Q_{k\Gg s\nu 2}+S_{j\Gb s\nu}], \eeq{y.11}
where sums over repeated indices are assumed. Again our theorem implies that with the boundary conditions
\eq{y.9} we have the boundary field equality that $\Bq(\Bx)\in\Bn(\Bx)\CD$  for all $\Bx\in\Md\GO$, where $\CD$  is
defined as before. 

Of course, $\BQ(\Bx)$ being divergence-free field in a simply connected two-dimensional region $\GO$ is equivalent to the existence of a $s$-component potential $\Bw(\Bx)$
such that 
\beq  \BQ(\Bx)=\BR^\perp\Grad\Bw.
\eeq{0_1}
The non-local linear constraints on the boundary flux $\Bq$ can then be rephrased as non-local linear constraints on $\Bw(\Bx)$ at the boundary $\Bx\in\Md\GO$.
The result that $\BQ(\Bx)\in\CD$ then implies $\Grad\Bw(\Bx)\in\CC\equiv\BR^\perp\CD$, and this not only constrains the tangential derivatives of the surface potential,
but also implies the additional boundary field equalities \eq{b.0f}.

Now consider the three dimensional case. Then a 
vector field $\Be(\Bx)$ is the gradient of a potential $u(\Bx)$ in $\GO$ if and only if the $3\times 3$ antisymmetric matrix valued field
\beq \BA(\Bx)=\bpm 0 & -e_3(\Bx) & e_2(\Bx) \cr 
                  e_3(\Bx) &  0 & -e_1(\Bx) \cr
                 - e_2(\Bx)& e_1(\Bx) & 0 \epm,
\eeq{ca.3}
is divergence-free. The requirement that $\BA(\Bx)$ be  antisymmetric can be viewed as constraining it to lie in a 3-dimensional subspace.
Knowing the boundary value of the flux $\Bn(\Bx)\cdot\BA(\Bx)$ (where $\Bn(\Bx)$ is the outward normal to $\Md\GO$) gives the tangential
components of $\Be(\Bx)$, which can be integrated to give $\Bu(\Bx)$ for $\Bx\in\Md\GO$. 
Thus, say, with a constitutive law $\mathbb{J}(\Bx)=\BL(\Bx)\mathbb{E}(\Bx)$ with $\mathbb{E}=\Grad\mathbb{U}$ where $\mathbb{U}(\Bx)$ is represented as a $3m^2$-component potential, 
and $\mathbb{J}(\Bx)$ is represented as a divergence-free $3\times 3m^2$ matrix-valued field, we can replace the $3\times 3m^2$ matrix-valued $\mathbb{E}$ by 
a $3\times 3\times 3m^2$ three-index field $\mathbb{A}$ that is antisymmetric in the first pair of indices. Then the constitutive relation
with the restriction that $\BL(\Bx)\in\CM$, and that $\BL(\Bx)$ is bounded and coercive, can be replaced by a nonlinear restriction $\BQ(\Bx)\in\CB$ involving 
a subset $\CB$ (independent of $\Bx$) of a non-linear manifold, and a divergence-free $3\times s$ matrix valued field $\BQ(\Bx)$, where $s=12m^2$
with components comprised of the $3\times 3m^2$ components of $\mathbb{J}(\Bx)$ and
the $3\times 9m^2$ components of $\mathbb{A}(\Bx)$. ($\CB$ is defined so that it ensures $\mathbb{A}$ is antisymmetric in the first pair of indices.)
The boundary information is contained in $\Bq(\Bx)=\Bn\cdot\BQ(\Bx)$ that is restricted to satisfy a non-local, linear, relation
of the form \eq{y.9} where the components of $\Bq_1$ represent the components of $\Bn\cdot\mathbb{J}$ and the components of $\Bq_2$ represent
the components of $\Bn\cdot\mathbb{A}$, and when appropriately defined the first $3m^2$ elements of $\Bq$ give $\Bq_1$, while the remaining last $9m^2$ elements of $\Bq$ give $\Bq_2$.
Our result that $\mathbb{J}(\Bx)-\BL_0\mathbb{E}(\Bx)\in \CS\BD$ for all $\Bx\in\GO$ then again implies $\BQ(\Bx)\in\CD$ , for some appropriately defined subspace $\CD$ .
We thus obtain the boundary field equality that $\Bq(\Bx)\in\Bn(\Bx)\cdot\CD$ .

Similarly, in three dimensions, a current field $\Bj(\Bx)$ satisfying $\Div\Bj(\Bx)=0$ can be expressed in terms of the curl of some vector potential $\Bg(\Bx)$,
or equivalently in terms of the antisymmetric part of $\Grad\Bg$. The values of $\Bn(\Bx)\cdot\Bj(\Bx)$ only provide partial information about the tangential
values of $\Grad\Bg(\Bx)$ at the boundary of $\Md\GO$ and these are insufficient to determine, by integration, $\Bg(\Bx)$ for $\Bx\in\Md\GO$. However, 
we can think of prescribing the boundary value of $\Bg(\Bx)$ for $\Bx\in\Md\GO$, which then allows one
to determine $\Bn(\Bx)\cdot\Bj(\Bx)$. More generally, with say $\mathbb{J}(\Bx)=\BL(\Bx)\mathbb{E}(\Bx)$ where $\mathbb{E}=\Grad\mathbb{U}$
and $\mathbb{U}(\Bx)$ is a $3m^2$-component potential, the $3\times 3m^2$ matrix-valued 
field $\mathbb{J}(\Bx)$ satisfying $\Div\mathbb{J}=0$ can be replaced
by the appropriate components of $\Grad\mathbb{G}(\Bx)$, for some $3\times 3m^2$ matrix-valued potential $\mathbb{G}(\Bx)$.
Instead of a relation like \eq{y.9} we obtain a restriction on the boundary potentials like
\beq \mathbb{U}-\BGL_{\GO^C}(\mathbb{G})=\mathbb{U}^0-\BGL_{\GO^C}(\mathbb{G}^0), \eeq{yyy.9}
for some appropriately defined exterior DtN map $\BGL_{\GO^C}$.
Thus, instead of relations involving the components of a divergence-free $3\times 12m^2$ matrix valued field $\BQ(\Bx)$,
one can express everything in terms of $\Grad\Bw$ where $\Bw(\Bx)$ is a $t$ component potential, where $t=12m^2$, comprised of the $3m^2$ elements of $\mathbb{U}$ and the $9m^2$ elements
of $\mathbb{G}(\Bx)$. The restriction that $\BQ(\Bx)\in\CB$ is then equivalent to a non-linear local restriction $\Grad\Bw\in\CA$. The result that
$\BQ(\Bx)\in\CD$ , can be rewritten as $\Grad\Bw\in\CC$ for some appropriately defined subspace $\CC$, and the boundary field equality 
that $\Bq(\Bx)\in\Bn(\Bx)\cdot\CD$  implies restrictions on the tangential derivatives of $\Bw$ at the surface $\Md\GO$, and also implies the additional
boundary field equality \eq{b.0f}.

Just as in two-dimensions, one can redefine $\CB$ (and hence $\CA$) to allow for a weaker restriction of the form \eq{y.11}.

\section{Conjectured extension of exact relations to some non-linear minimization problems}	
\setcounter{equation}{0}
The canonical equations \eq{I.1} are the Euler-Lagrange equations for the minimizer $\underline{\BE}=\BE$ of the
following minimization
problem
\beq \inf_{\underline{\BE}\in\CE}\lang\BL\underline{\BE}-2\Bh,\underline{\BE}\rang_\CH.
\eeq{m.1} 
Similarly the extended equations \eq{extgov} are the Euler-Lagrange equations  for the minimizer $\underline{\mathbb{E}}=\mathbb{E}$  
of the following minimization problem
\beq \inf_{\underline{\mathbb{E}}\in\mathfrak{E}}\lang\BL\underline{\mathbb{E}}+2(\BL-\BL_0)\mathbb{S},\underline{\mathbb{E}}\rang_{\mathfrak{H}},
\eeq{m.2}
where the inner product on $\mathfrak{H}=\mathfrak{E}\oplus\mathfrak{J}$ is defined by \eq{ip}.
Now for each point $\Bx$ let $\CG(\Bx)$ denote a subset of the manifold $\CM$ that is closed under homogenization.
This is guaranteed if given any $\By$-periodic function $\BL(\By)\in\CG(\Bx)$ for all $\By\in\RR^d$, then the associated effective tensor $\BL^*$
also lies in $\CG(\Bx)$. (That it suffices to consider periodic functions was established in \cite{Allaire:2000:SOH,Raitums:2001:LRC}). We also assume 
that $\CG(\Bx)$ varies smoothly with $\Bx$. Then consider the following double minimization problem
\beq \inf_{\begin{matrix}\BL\\\BL(\Bx)\in\CG(\Bx),\,\forall\Bx\end{matrix}}\inf_{\mathbb{E}\in\mathfrak{E}}\lang\BL\mathbb{E}+2(\BL-\BL_0)\mathbb{S},\mathbb{E}\rang_{\mathfrak{H}}.
\eeq{m.3}
Following the ideas of Kohn \cite{Kohn:1991:RDW} we switch the order of taking infimums to get the nonlinear minimization problem:
\beq \inf_{\underline{\mathbb{E}}\in\mathfrak{E}}\int_{\BR^d}W(\Bx,\underline{\mathbb{E}}(\Bx))\,d\Bx,
\eeq{m.4}
where for all $\Bx\in\RR^d$ and all $\BB\in L(\cal T)$,
\beq  W(\Bx,\BB)=\inf_{\BA\in\CG(\Bx)}\lang\BA\BB+2(\BA-\BL_0)\mathbb{S}(\Bx),\BB\rang_{L(\cal T)}.
\eeq{m.5} 
One can view the nonlinear ``energy function'' $W(\Bx,\BB)$ as being the infimum of a continuous set of
quadratic functions. If $\mathbb{S}(\Bx)$ is smooth one expects the problem \eq{m.4} to have a smooth minimizer
$\underline{\mathbb{E}}=\mathbb{E}$. Otherwise, if a minimizing sequence $\underline{\mathbb{E}}_j, j=1,2,\ldots$ develops
fine-scale oscillations, it would be indicative that associated sequence of tensor fields $\BL(\Bx)=\BL_j(\Bx)$ that are a minimizing
sequence for \eq{m.3} also develops fine-scale microstructure as $j$ increases. The development
of such fine-scale microstructure indicates that composite materials are needed in the construction, having some homogenized effective tensor
$\BL^*(\Bx)$.  However because $\CG(\Bx)$ is closed under homogenization we should be able to just directly use
materials in $\CG(\Bx)$ whose tensor matches that of the desired effective tensor $\BL^*(\Bx)$. In this way fine-scale oscillations in $\BL_j(\Bx)$ and
hence $\underline{\mathbb{E}}_j(\Bx)$ (with $j$ large) can be removed.  

This argument strongly suggests that there should be a minimizer of \eq{m.4} and that minimizer 
$\underline{\mathbb{E}}(\Bx)=\mathbb{E}(\Bx)$ should satisfy the nonlinear 
Euler-Lagrange equations:
\beq \mathbb{J}(\Bx)=\frac{\Md W(\Bx,\BB)}{\Md\BB}\bigg|_{\BB=\mathbb{E}(\Bx)},\quad
\mathbb{J}\in\mathfrak{J},\quad \mathbb{E}\in\mathfrak{E}. 
\eeq{m.6}
Then from  \eq{m.5} it follows that there exists some $\BL(\Bx)\in\CG(\Bx)\subset\CM$ such that
\beq \mathbb{J}(\Bx)=\BL(\Bx)\mathbb{E}(\Bx)+(\BL(\Bx)-\BL_0)\mathbb{S}(\Bx).
\eeq{m.7}
Indeed, a tangent plane to a minimum of a set of quadratic functions must be tangent to at least one of them.
If $\mathbb{S}(\Bx)$ entering \eq{m.5} takes values in $\CS\BD$ for some nonsingular $\BD: \CT\to\CT$ we conclude that 
\beq \mathbb{P}(\Bx)=\mathbb{J}(\Bx)-\BL_0\mathbb{E}(\Bx)
\eeq{m.8}
necessarily takes values in $\CK\BD$. Also if all tensors in $\CG(\Bx)$ are self-adjoint for all $\Bx$, then if 
$\mathbb{S}(\Bx)$ takes values in $(\CK\BD)^\perp$, $\mathbb{P}(\Bx)$ necessarily takes values in $(\CS\BD)^\perp$.
\section{Conclusions}	
\setcounter{equation}{0}
We have laid the foundations of the theory of exact relations for wide classes of linear partial differential equations, 
where the coefficients are position dependent and satisfy appropriate non-linear constraints, and have
derived exact relations satisfied by their Green's functions. Similar to the theory of exact relations in composites, 
it all boils down to finding subspaces $\CK\subset L(\cal T)$ satisfying
\eq{0.1} and associated subspaces $\CS\subset L(\cal T) $ such that $\BA\cal S \subset \cal K$ for all $\BA\in\cal K$. 
The algebraic search for such subspaces $\CK$ or, more precisely the subspaces $\CK_0$ consisting of symmetric tensors in $\CK$,
has been intensively studied by Grabovsky, Sage, and subsequent coworkers in the case when
$\BGY(\Bk)$ is only a function of $\Bk/|\Bk|$. Their progress is summarized in \cite{Grabovsky:2004:AGC,Grabovsky:2016:CMM}. To apply
their results to our setting a relatively easy task remains, namely to identify the subspaces $\CK$ and $\CS$ associated with each $\CK_0$.
\aj{We caution, though, that some of these results are for equations such as thermoelasticity where the field $\BE(\Bx)$ has constant
components, independent of $\Bx$, such as the temperature increment $\Gt$. Then $\BE(\Bx)$ is not square integrable, and our analysis does not apply
in its current form. It may apply if we expand the constitutive law and treat those terms involving constant fields as source terms.
Thus, for example, consider the thermoelastic equation
\beq \BGe(\Bx)=\BCS(\Bx)\BGs(\Bx)+\BGa(\Bx)\Gt,\quad \BGe=[\Grad\Bu+(\Grad\Bu)^T]/2, \quad\Div\BGs=0, \eeq{c.0}
where $\BGe(\Bx)$ is the strain, $\Bu(\Bx)$ is the displacement field, $\BGs(\Bx)$ is the stress, $\Gt=T-T_0$ is the change in temperature $T$ 
measured from some base temperature $T_0$,
$\BCS(\Bx)$ is the compliance tensor (inverse elasticity tensor) and $\BGa(\Bx)$ is the tensor of thermal expansion. Provided $\BGa(\Bx)$
is square integrable, we can treat $\Bh(\Bx)=-\BGa(\Bx)\Gt$ as our source term, $\BGs(\Bx)$ as the field $\BE(\Bx)$, and $\BGe(\Bx)$ as the field 
$\BJ(\Bx)$. Of course, there could be additional source terms, not just those arising from the thermal expansion.}

For many other equations of interest, such as wave equations in lossy media, $\BGY(\Bk)$ for $\Bk\ne 0$ has the form
\beq \BGY(\Bk)=\frac{1}{q(\Bk)}\sum_{j=0}^p|k|^j\BGY_j(\Bk/|\Bk|), \eeq{c.1}
where $q(\Bk)$ is scalar valued, and so \eq{0.1} holds for all $\Bk\ne 0$ if and only if for all unit vectors $\Bn$, and for $j=0,1,\ldots,p$,
\beq \BA_1\BGY_j(\Bn)\BA_2\in\CK,~~{\rm for~all~}\BA_1,\BA_2\in\CK.
\eeq{c.2}
Thus we are confronted with algebraic questions that are essentially the same as those investigated by Grabovsky, Sage, and subsequent coworkers.
The next step will be to search for specific examples of physical relevance, beyond those encountered in the theory of exact relations for composites.
\aj{Irrespective of whether such exact relations exist for wave-equations in lossy media it is to be emphasized that the current theory already applies
to many of the examples studied by Grabovsky, Sage and coworkers as summarized in the book \cite{Grabovsky:2016:CMM}}.

We also made large strides in developing the theory of ``boundary field equalities'', and provided for the first time a general theory for exact
relations satisfied by the DtN map for bodies containing appropriate inhomogeneous media. Again, examples are needed to illuminate the theory
and bring our results from an abstract setting to practice. Our boundary field equalities relied heavily on the existence of an appropriate constitutive law inside the body,
or at least by imposing non-linear constraints on fields so they could be related by an appropriate constitutive law, where the tensor $\BL(\Bx)$ entering the constitutive law
depends upon the local fields. An interesting question is: Are there boundary field equalities, beyond the standard conservation laws, that do not require this of the
interior fields? For example, in the setting of Section 8 one may ask about the existence of sets $\CB$ and appropriate boundary conditions in the two-dimensional
case where $\BQ$ is a divergence-free $2\times s$ matrix valued field and $s$ is not the square of an even integer, and in the three dimensional case
where $\BQ$ is a divergence-free $2\times s$ matrix valued field and $s/3$ is not the square of an even integer. (Of course, one may add to $\BQ$ an 
arbitrary number of ``dummy'' divergence-free columns that do not participate in the analysis, so one would want to exclude such trivial manipulations).

In our analysis we made heavy use of the tensor $\BL_0$, but the Green's function $\BG(\Bx,\Bx^0)$ is independent of $\BL_0$. We could have replaced
$\BL_0$ by any other tensor $\BL_0'$ on the manifold $\CM=W_\BM^{-1}(\CK_0)$ (where $\CK_0$ consists
of all self-adjoint maps in $\CK$) and the analysis would have carried through. This raises the question, which we have not
explored, as to whether this different choice of $\BL_0$ would yield new constraints on $\BG(\Bx,\Bx^0)$, or just recover those obtained with the original choice
of $\BL_0$.

\section*{Acknowledgements}
GWM is grateful to the National Science Foundation for support through the Research Grants DMS-1211359 and DMS-1814854. Both authors thank the Institute for Mathematics and its Applications
at the University of Minnesota for hosting their visit there during the Fall 2016 where this work was initiated as part of the program on Mathematics and Optics. Yury Grabovsky and the referees
are thanked for their comments which led to significant improvements of the manuscript.

\section*{Conflict of Interest}
On behalf of all participating authors, the corresponding author states that there is no conflict of interest. 
\bibliographystyle{plain}
\bibliography{/home/milton/tcbook,/home/milton/newref}
\section*{Appendix: Some physical equations that can be expressed in the required canonical form}

Here we give examples of some physical equations that can be expressed in the required canonical form. The examples are
by no means comprehensive: for further examples, see \cite{Milton:2002:TOC} Chap.2, \cite{Milton:2016:ETC} Chap.1, \cite{Milton:2013:SIG,Milton:2015:ATS} 
and the appendix of \cite{Milton:2018:NRF}. In the equations that follow we omit the source terms. We emphasize that if one is interested in
exact relations satisfied by the DtN-map and the associated boundary field equalities, then it is not necessary that the source terms have a
physical significance.

The simplest canonical equations are those of electrical conductivity 
\beq \underbrace{\Bj(\Bx)}_{\BJ(\Bx)}=\underbrace{\BGs(\Bx)}_{\BL(\Bx)}\underbrace{\Be(\Bx)}_{\BE(\Bx)},\quad \quad \Div\Bj=0, \quad\quad \Be=-\Grad V=0,
\eeq{a1}
where $\Bj(\Bx)$ and $\Be(\Bx)$ are the electrical current and electric field and $V(\Bx)$ is the electrical potential. The boundary fields
$\Md\BJ$ and $\Md\BE$ are the flux $\Bn\cdot\Bj(\Bx)$ and boundary voltage $V(\Bx)$, respectively, with  $\Bx\in\Md\GO$ and $\Bn$ being the outwards normal to $\GO$.
As displayed  in the table at the beginning of section 2.1 in \cite{Milton:2002:TOC} 
(adapted from one of Batchelor \cite{Batchelor:1974:TPT}) the equations for 
dielectrics, magnetostatics, heat conduction, particle diffusion, flow in porous media, and antiplane elasticity
all take the same form as \eq{a1} and so any analysis applicable to \eq{a1} applies to them as well. 

Another important example is that of linear elasticity,
\beq \underbrace{\BGs(\Bx)}_{\BJ(\Bx)}=\underbrace{\BCC(\Bx)}_{\BL(\Bx)}\underbrace{\BGe(\Bx)}_{\BE(\Bx)},~~~~\Div\BGs=0,~~~\BGe=[\Grad\Bu+(\Grad\Bu)^T]/2,
\eeq{a3}
where $\BGs(\Bx)$ (not to be confused for the conductivity tensor field) is the stress, $\BGe(\Bx)$ is the strain, $\Bu(\Bx)$ is the
displacement, and $\BCC(\Bx)$ is the elasticity tensor field. The boundary fields
$\Md\BJ(\Bx)$ and $\Md\BE(\Bx)$ are the traction $\Bn\cdot\BGs(\Bx)$ and boundary displacement field $\Bu(\Bx)$, respectively.

It is also possible to have equations that couple fields together, such as the magnetoelectric equations,
\beq   
\underbrace{\bpm \Bd \cr \Bb \epm}_{\BJ(\Bx)}=\underbrace{\bpm \BGve & \BGb \cr \BGb^T & \BGm \epm}_{\BL(\Bx)}\underbrace{\bpm \Be \cr \Bh\epm}_{\BE(\Bx)},\quad\Div\Bd=\Div\Bb=0,\quad
\Be=-\Grad V,\quad \Bh=-\nabla\psi,
\eeq{a4}
where $\Bd$ and $\Bb$ are the electric displacement field and magnetic induction, $\Be$ and $\Bh$ are the electric and magnetic fields,
$V$ and $\psi$ are the electric potential and magnetic scalar potential (assuming there are no free currents),
$\BGve$ is the free-body electrical permittivity (with $\Bh=0$), $\BGb(\Bx)$ is the second-order magnetoelectric
coupling tensor, $\BGm(\Bx)$ is the free-body magnetic permeability (with $\Be=0$). The boundary fields
$\Md\BJ(\Bx)$ and $\Md\BE(\Bx)$ are then the flux pair \aj{$(\Bn\cdot\Bd(\Bx),\Bn\cdot\Bb(\Bx))$} and the
potential pair $(V(\Bx),\psi(\Bx))$, respectively, with $\Bx\in\Md\GO$.

Thermoelectricity also takes this form, but one has to be careful in defining the fields to ensure that the associated tensor
$\BL(\Bx)$ is symmetric (see, e.g., Section 2.4 in \cite{Milton:2002:TOC}).

Fields that are coupled together need not have the same tensorial rank, an example being the equations of piezoelectricity,
\beq \underbrace{\bpm\BGe \cr \Bd\epm}_{\BJ(\Bx)}=\underbrace{\bpm\BCS & \BCD \cr \BCD^T & \BGve\epm}_{\BL(\Bx)}
\underbrace{\bpm \BGs \cr \Be \epm}_{\BE(\Bx)},
\eeq{a5.1}
where $\BCS(\Bx)$ is the compliance tensor under short-circuit boundary conditions
(i.e.,  with $\Be=0$), $\BCD(\Bx)$ is the piezoelectric stress coupling tensor,
and $\BGve(\Bx)$ is the free-body dielectric tensor (i.e.,  with $\BGj=0$). The
strain field $\BGe$, electric displacement field $\Bd$, stress field $\BGs$, and
electric field $\Be$ satisfy the usual differential constraints:
\beq \BGe  =  [\Grad\Bu+(\Grad\Bu)^T]/2,\quad \Div\Bd  =  0,\quad \Div\BGs  =  0,\quad \Be  = -\Grad V. \eeq{a5.2}
Since the stresses and strains are symmetric matrices, $\BCD$ is a third-order
tensor that maps vectors to symmetric matrices. The boundary fields $\Md\BJ(\Bx)$ and $\Md\BE(\Bx)$ are
the displacement, flux pair $(\Bu(\Bx),\Bn\cdot\Bd(\Bx))$ and the traction, voltage pair $(\Bn\cdot\BGs,V)$, respectively.

Of course, more than two fields can be coupled together. Thus, by combining a piezoelectric material and a magnetostrictive material 
in a composite we can obtain a material where there is coupling between electric fields, elastic fields, and magnetic fields. 

For fields varying in time at constant frequency $\Go$, with wavelengths and attenuation
lengths much bigger than the size of the body under consideration, the quasistatic equations are applicable. For dielectrics these take 
the same form as \eq{a1}:
\beq  \Bd(\Bx)=\BGve(\Bx)\Be(\Bx),\quad \quad \Div\Bj=0, \quad\quad \Be=-\Grad V=0,
\eeq{a6}
where everything is now complex valued: $\Bd(\Bx)$ and $\Be(\Bx)$ are the complex valued electrical displacement field and electric field, 
$V(\Bx)$ is the complex valued electrical potential, and $\BGve(\Bx)$ is the complex valued electrical permittivity. (The
physical displacement field, electric field, and potential are the real parts of $\Bd(\Bx)e^{-i\Go t}$, $\Be(\Bx)e^{-i\Go t}$,
and $V e^{-i\Go t}$ respectively). Let us set
\beq \Bd=\Bd'+i\Bd'',\quad \Be=\Be'+i\Be'',\quad V=V'+iV'',\quad \BGve=\BGve'+\BGve'', \eeq{a7}
where the primed fields denote the real parts, while the doubled primed fields denote the imaginary parts. Physically, $\BGve''(\Bx)$
is associated with electrical energy loss into heat and is positive semidefinite. Assuming it is positive definite and that an
inverse $[\BGve'']^{-1}$ exists, substitution of \eq{a7} in \eq{a6}, followed by suitable manipulation,
 gives the equivalent coupled field equations of Gibiansky and Cherkaev \cite{Cherkaev:1994:VPC}:
\beq \underbrace{\bpm \Be'' \cr \Bd'' \cr \epm}_{\BJ(\Bx)}=
\underbrace{\bpm
[\BGve'']^{-1} & [\BGve'']^{-1} \BGve' \cr
\BGve'[\BGve'']^{-1} & \BGve''
+ \BGve'[\BGve'']^{-1} \BGve' \epm}_{\BL(\Bx)}\underbrace{\bpm-\Bd' \cr \Be' \cr\epm}_{\BE(\Bx)},
\eeq{a8}
Clearly $\BL$ is real and symmetric and by inspection of the quadratic form associated with $\BL$ one sees that it is positive definite.
Now $\Md\BJ$ consists of the voltage, flux pair $(V'',\Bn\cdot\Bd'')$ while $\Md\BE$ consists of the flux, voltage pair $(-\Bn\cdot\Bd',V')$.
As Gibiansky and Cherkaev show, similar manipulations can be done for viscoelasticity in the quasistatic limit where the equations
have the form \eq{a3}, but with all fields being complex. More generally, the Gibiansky-Cherkaev approach can be applied
to equations where the tensor entering the constitutive law is not self-adjoint, but its self-adjoint part is positive definite,
to an equivalent form where the tensor $\BL(\Bx)$ entering the constitutive law is self-adjoint and positive definite \cite{Milton:1990:CSP}
(see also section 13.4 of \cite{Milton:2002:TOC}): such
manipulations can be applied, for example, to electrical conduction in the presence of a magnetic field where, due to the Hall effect,
the conductivity tensor $\BGs(\Bx)$ entering \eq{a1} is not symmetric.

Wave equations, can be expressed in the form \eq{4} with an identity
like \eq{5} holding. For example, at fixed frequency $\Go$ with a $e^{-i\Go t}$ time dependence, as recognized in \cite{Milton:2009:MVP}
the acoustic equations, with $P(\Bx)$ the (complex) pressure, $\Bv(\Bx)$ the (complex) velocity, $\BGr(\Bx,\Go)$ the effective mass density matrix, and $\Gk(\Bx,\Go)$ the bulk modulus, take the form
\beq \underbrace{\begin{pmatrix}-i\Bv \\ -i\Div\Bv \end{pmatrix}}_{\BJ(\Bx)}
=\underbrace{\begin{pmatrix}-(\Go\BGr)^{-1} & 0 \\ 0 & \Go/\Gk\end{pmatrix}}_{\BL(\Bx)}\underbrace{\begin{pmatrix}\Grad P \\ P\end{pmatrix}}_{\BE(\Bx)},
\eeq{5.A}
(and $\Md\BE$ and $\Md\BJ$ can be identified with the boundary values of $P(\Bx)$ and $\Bn\cdot\Bv(\Bx)$ at $\Md\GO$, respectively).
Here we allow for effective mass density matrices that, at a given frequency, can be anisotropic and complex valued as may be the case in metamaterials 
\cite{Schoenberg:1983:PPS,Willis:1985:NID,Milton:2006:CEM,Milton:2007:MNS}.
Maxwell's equations, with $\Be(\Bx)$ the (complex) electric field, $\Bh(\Bx)$ the (complex) magnetizing field, $\BGm(\Bx,\Go)$ the magnetic permeability,
$\BGve(\Bx)$ the electric permittivity, take the form
\beq \underbrace{\begin{pmatrix}-i\Bh \cr i\Curl\Bh\end{pmatrix}}_{\BJ(\Bx)}
=\underbrace{\begin{pmatrix}-{[\Go\BGm]}^{-1} & 0 \\ 0 & \Go\BGve \end{pmatrix}}_{\BL(\Bx)}
\underbrace{\begin{pmatrix}\Curl\Be \\ \Be\end{pmatrix}}_{\BE(\Bx)},
\eeq{5.B}
(and $\Md\BE$ and $\Md\BJ$ can be identified with the tangential values of $\Be(\Bx)$ and $\Bh(\Bx)$ at $\Md\GO$, respectively).
The linear elastodynamic equations, with $\Bu(\Bx)$ the (complex) displacement, $\BGs(\Bx)$ the (complex) stress, $\BCC(\Bx,\Go)$ the elasticity tensor,
$\BGr(\Bx,\Go)$ the effective mass density matrix,  take the form
\beq \underbrace{\begin{pmatrix} -\BGs/\Go \\ -\Div\BGs/\Go\end{pmatrix}}_{\BJ(\Bx)} 
=\underbrace{\begin{pmatrix}-\BCC/\Go & 0 \\ 0 & \Go\BGr\end{pmatrix}}_{\BL(\Bx)}\underbrace{\begin{pmatrix}[\Grad \Bu+(\Grad \Bu)^T]/2 \\ \Bu\end{pmatrix}}_{\BE(\Bx)},
\eeq{5.C}
(and $\Md\BE$ and $\Md\BJ$ can be identified with the values of $\Bu(\Bx)$ and the traction $\Bn\cdot\BGs(\Bx)$ at $\Md\GO$, respectively).
The preceeding three equations have been written in this form so $\Imag\BL(\Bx)\geq 0$
when $\Imag\Go\geq 0$, where complex frequencies have the physical meaning of the solution increasing exponentially in time.
Under assumptions that the material moduli are lossy, or that 
the frequency $\Go$ is complex with positive imaginary part, one can easily manipulate them into equivalent forms similar to the Gibiansky-Cherkaev form
in \eq{a8} with a positive semidefinite tensor entering the constitutive law \cite{Milton:2009:MVP,Milton:2010:MVP}. Of course, the boundary fields $\Md\BE$ and
$\Md\BJ$ then need to be appropriately redefined.

For thin plates, the dynamic plate equations at constant frequency
can be written in the form
\beq
\underbrace{\begin{pmatrix}
i\BM\\
\nabla\cdot(\nabla\cdot\BM)
\end{pmatrix}}_{\BJ(\Bx)}
=
\underbrace{\begin{pmatrix}
-\BCD(\Bx)/\Go & 0 \\
0 & h(\Bx)\Go\Gr(\Bx)
\end{pmatrix}}_{\BL(\Bx)}
\underbrace{\begin{pmatrix}
\nabla\nabla v
\\
i v
\end{pmatrix}}_{\BE(\Bx)}.
\eeq{*6} 
Here $\BM(\Bx,t)$ is the (complex) bending moment tensor,
$\BCD(\Bx)$ is the fourth-order tensor
of plate rigidity coefficients, $h(\Bx)$ is the plate thickness, $\Gr(\Bx)$ is the density,
and $v=\Md w/\Md t$ is the velocity of the (complex) vertical deflection $w(\Bx,t)$ of the plate.
Note that the matrix $\BL(\Bx)$ has positive definite imaginary part when $\Go$ has positive imaginary part.
$\Md\BE$ can be identified with the boundary values of the pair $(\Grad v,v)$ while
$\Md\BJ$ can be identified with the boundary values of the pair $(\BM\Bn,(\Div\BM)\cdot\Bn)$, in which $\Bn$
is the outwards normal to $\Md\GO$. 
Again, when the material moduli are lossy, or 
the frequency $\Go$ is complex with positive imaginary part, this can be manipulated into the Gibiansky-Cherkaev form
in \eq{a8} with a positive semidefinite tensor entering the constitutive law, and with appropriately redefined boundary fields.

Further examples of wave equations at constant frequency that can be represented in the required form are given 
in the appendix of \cite{Milton:2018:NRF}.
\end{document}